\documentclass[english]{article}
\usepackage[T1]{fontenc}
\usepackage[latin9]{inputenc}
\usepackage{geometry}
\usepackage{amsmath}
\usepackage{amsthm}
\usepackage{amssymb}
\usepackage{stmaryrd}
\usepackage{graphicx}

\makeatletter

\providecommand{\tabularnewline}{\\}

\numberwithin{equation}{section}
\numberwithin{figure}{section}
\theoremstyle{plain}
\newtheorem{thm}{\protect\theoremname}
\theoremstyle{plain}
\newtheorem{lem}[thm]{\protect\lemmaname}
\theoremstyle{plain}
\newtheorem{prop}[thm]{\protect\propositionname}
\theoremstyle{remark}
\newtheorem{rem}[thm]{\protect\remarkname}
\theoremstyle{definition}
\newtheorem{defn}[thm]{\protect\definitionname}
\theoremstyle{plain}
\newtheorem{cor}[thm]{\protect\corollaryname}

\makeatother

\usepackage{babel}
\providecommand{\corollaryname}{Corollary}
\providecommand{\definitionname}{Definition}
\providecommand{\lemmaname}{Lemma}
\providecommand{\propositionname}{Proposition}
\providecommand{\remarkname}{Remark}
\providecommand{\theoremname}{Theorem}

\begin{document}
\global\long\def\ve{\varepsilon}%
\global\long\def\R{\mathbb{R}}%
\global\long\def\Rn{\mathbb{R}^{n}}%
\global\long\def\Rd{\mathbb{R}^{d}}%
\global\long\def\E{\mathbb{E}}%
\global\long\def\P{\mathbb{P}}%
\global\long\def\bx{\mathbf{x}}%
\global\long\def\vp{\varphi}%
\global\long\def\ra{\rightarrow}%
\global\long\def\smooth{C^{\infty}}%
\global\long\def\Tr{\mathrm{Tr}}%
\global\long\def\bra#1{\left\langle #1\right|}%
\global\long\def\ket#1{\left|#1\right\rangle }%
\global\long\def\Re{\mathrm{Re}}%
\global\long\def\Im{\mathrm{Im}}%
\global\long\def\bsig{\boldsymbol{\sigma}}%
\global\long\def\btau{\boldsymbol{\tau}}%
\global\long\def\bmu{\boldsymbol{\mu}}%
\global\long\def\bx{\boldsymbol{x}}%
\global\long\def\bups{\boldsymbol{\upsilon}}%
\global\long\def\bSig{\boldsymbol{\Sigma}}%
\global\long\def\bt{\boldsymbol{t}}%
\global\long\def\bs{\boldsymbol{s}}%
\global\long\def\by{\boldsymbol{y}}%
\global\long\def\brho{\boldsymbol{\rho}}%
\global\long\def\ba{\boldsymbol{a}}%
\global\long\def\bb{\boldsymbol{b}}%
\global\long\def\bz{\boldsymbol{z}}%
\global\long\def\bc{\boldsymbol{c}}%
\global\long\def\balpha{\boldsymbol{\alpha}}%
\global\long\def\bbeta{\boldsymbol{\beta}}%
\global\long\def\bu{\boldsymbol{u}}%
\global\long\def\bv{\boldsymbol{v}}%

\title{{\Large{}Multiscale interpolative construction of quantized tensor
trains}}
\author{Michael Lindsey\\
{\small{}UC Berkeley}{\footnotesize{}}\\
{\small{}$\texttt{lindsey@berkeley.edu}$}}
\maketitle
\begin{abstract}
Quantized tensor trains (QTTs) have recently emerged as a framework
for the numerical discretization of continuous functions, with the
potential for widespread applications in numerical analysis. However,
the theory of QTT approximation is not fully understood. In this work,
we advance this theory from the point of view of multiscale polynomial
interpolation. This perspective clarifies why QTT ranks decay with
increasing depth, quantitatively controls QTT rank in terms of smoothness
of the target function, and explains why certain functions with sharp
features and poor quantitative smoothness can still be well approximated
by QTTs. The perspective also motivates new practical and efficient
algorithms for the construction of QTTs from function evaluations
on multiresolution grids.
\end{abstract}

\section{Introduction}

Quantized tensor trains (QTTs) \cite{QTT} have been proposed as a
tool for the discretization of functions of one or several continuous
variables. QTTs offer an unconventional point of view compared to
classical frameworks such as ordinary grid-based discretization and
basis expansion. Indeed, the QTT format is motivated by identifying
functions of a continuous variable with tensors via binary decimal
expansion of the argument, then leveraging the tensor network format
known as the tensor train (TT) \cite{OseledetsTyrtyshnikov2009,oseledets2011tensortrain}
or matrix product state (MPS) \cite{Fannes_MPS,Klumper_MPS,White1992,PerezGarcia_MPS}.

Although the QTT format was forwarded over a decade ago in 2011 \cite{QTT},
it has seen a recent surge of interest, motivated by applications
to fluid mechanics \cite{QTT_fluids}, plasma physics \cite{ye2022,ye2023quantized},
quantum many-body physics \cite{Shinaoka2023qtt}, quantum chemistry
\cite{KhoromskaiaKhoromskijSchneider2011,jolly2023tensorized}, and
Fokker-Planck equations \cite{doi:10.1137/120864210}. Much of the
promise of QTTs derives from the fact that certain operations such
as convolution \cite{Kazeev2013} and the discrete Fourier transform
\cite{superfastFFT} are known to be efficient in the QTT format.
In recent work \cite{ChenStoudenmireWhite2023}, the discrete Fourier
transform was in fact demonstrated to have low rank as a matrix product
operator (MPO).

However, current understanding of which functions can be approximated
with QTTs is incomplete. Existing analysis often proceeds by representing
a function as a sum of building blocks \cite{doi:10.1137/120864210},
such as complex exponentials, known to have low rank. From this point
of view, a picture emerges in which smoothness controls the QTT rank,
since, e.g., a function with rapidly decaying Fourier coefficients
can be written as a sum of only a few complex exponentials. Similar
analysis has been pursued based on the replacement of a function with
a polynomial interpolation or approximation \cite{ShiTownsend2021}.

However, this type of analysis cannot fully explain the approximation
power of QTTs. First, it cannot explain why the ranks tend to decay
with increasing depth in the QTT, since rank bounds derived from summation
of complex exponentials or polynomial approximation are uniform across
all depths. Relatedly, we will see that even the global rank bounds
derived from such approaches are suboptimal. Second, it cannot explain
why certain functions with sharp features (i.e., poor quantitative
smoothness) often admit low-rank representations as QTT. We comment
that some recent work \cite{MazenNouy1,MazenNouy2} has studied quantized
tensor networks through the lens of Besov spaces, though the goals
and focus of that work are quite different from our own.

In this work, we analyze the construction of QTTs from the point of
view of \emph{multiscale} polynomial interpolation. Our analysis addresses
the open questions highlighted above.

First, we show quantitative rank bounds in terms of the smoothness
of the target function which decay with depth. In fact, our decaying
rank bounds imply that $\Omega$-bandlimited functions have QTT ranks
uniformly bounded by about $\sqrt{\Omega}$. (See Corollary \ref{cor:sqrt}
below for a precise statement.) This result is particularly striking
because, by expressing an $\Omega$-bandlimited function as a linear
combination of $O(\Omega)$ complex exponentials, we naively expect
a uniform rank bound of only $O(\Omega)$.

Second, we reach the stylized conclusion that functions which are
well approximated in a multiresolution polynomial basis are well represented
as QTTs. This conclusion is formalized in Theorem \ref{thm:multires}.

In addition to theoretical insight, the interpolative perspective
motivates practical rank-revealing algorithms for the construction
of QTTs from function evaluations on multiresolution grids. Existing
approaches based on Fourier truncation and separability assumptions
\cite{Oseledets2013constructiverep} may suffer from costs greater
than is necessitated by the true underlying rank of the QTT and moreover
cannot be extended to account for multiresolution structure. In fact,
even for the construction of exact polynomials as QTTs, to our knowledge
our approach defines the first numerically stable method based on
function evaluations, since the known exact construction \cite{Oseledets2013constructiverep}
relies on the expansion of a polynomial in a monomial basis, which
may involve extremely large coefficients for high-order polynomials.

Meanwhile, compared to tensor cross interpolation (TCI) \cite{oseledets2010ttcross},
the most celebrated generic algorithm for the construction of TTs
from black-box evaluations, our approach can yield significant efficiency
gains in terms of the number of function evaluations, and it performs
robustly even when TCI fails to converge. Meanwhile, all function
evaluations in our scheme are embarrassingly parallel, by contrast
to the evaluations in TCI. (We comment, however, that the applicability
of TCI to the construction of TTs that are not QTTs is much wider,
and our approaches rely on the specific structure of the QTT setting.)

Finally, we comment that our presentation makes it clear how to pass
back and forth between the representation of a function as a QTT and
its evaluation on a multiresolution grid of interpolating points.
This perspective opens the door to hybrid algorithms that combine
the QTT format with a format based on evaluations on a multiresolution
interpolating grid, similar to a discontinuous Galerkin or multiwavelet
representation. Certain operations such as convolution or Fourier
transformation may be convenient in the QTT format, while others such
as pointwise function composition may be more convenient in the multiresolution
interpolating grid format. Moreover, passing back and forth between
these formats using fast linear algebra (exploiting the sparse and,
in the multivariate case, Kronecker-factorized structure of our tensor
cores) may ultimately be more efficient than fully operating within
the QTT format, where cubic costs of key algorithms in the QTT ranks
may become prohibitively large. We highlight the practical investigation
of this point of view as a topic for further work.

\subsection{Outline}

In Section \ref{sec:Preliminaries}, we review the idea of QTTs. In
Section \ref{sec:decaying}, we prove rank bounds for QTTs that decay
with depth. In Section \ref{sec:interpolative}, we present our practical
approach for QTT construction, first in its most basic form but then
with several extensions that improve the practical efficiency. (See
the beginning of Section \ref{sec:interpolative} for a detailed outline
of each subsection.) In Section \ref{sec:invert}, we explain how
to `invert' the QTT construction, recovering function evaluations
on interpolating grids from a given QTT. In Section \ref{sec:sharp},
we present a multiresolution extension of our interpolative construction,
as well as theory that explains the representability of certain functions
with sharp features as QTT. In Section \ref{sec:multivariate}, we
explain how the preceding discussion extends to the setting of multivariate
functions, where several conventions for ordering tensor indices are
possible. In Section \ref{sec:numerical}, we conclude with numerical
experiments illustrating our theory and practical algorithms.

\subsection{Acknowledgments}

The author is grateful to Jielun Chen and Sandeep Sharma for stimulating
discussions on QTTs.

\section{Preliminaries \label{sec:Preliminaries}}

Consider a function $f:[0,1]\ra\R$. The idea of the quantized representation
is that we can place the variable $x\in[0,1]$ in bijection with sequences
of the form $\sigma_{1},\sigma_{2},\ldots$, where each $\sigma_{k}\in\{0,1\}$,
via the identification 
\begin{equation}
x=\sum_{k=1}^{\infty}2^{-k}\sigma_{k}=0.\sigma_{1}\sigma_{2}\sigma_{3}...,\label{eq:ident}
\end{equation}
 where the entries in the expression at right indicate binary decimal
expansions.

We choose a depth $K$ at which to truncate the decimal expansion,
so the identification 
\[
x\leftrightarrow(\sigma_{1},\ldots,\sigma_{K})
\]
 is a bijection between the dyadic grid $\mathcal{D}_{K}:=\left(2^{-K}\,\mathbb{Z}\right)\cap[0,1)$
and the set $\{0,1\}^{K}$. Based on this identification, we can in
turn identify functions $f:\mathcal{D}_{K}\ra\R$ with tensors $T\in(\R^{2})^{K}\simeq\R^{2}\times\cdots\times\R^{2}$
($K$ factors) via the relation
\[
f(x)=T(\sigma_{1},\ldots,\sigma_{K}).
\]
 Under such an identification, we will refer to $T$ as the \textbf{\emph{quantized
tensor}} representation of $f$.

The tensor $T$ can be viewed as a tensor train (TT) or matrix product
state (MPS) if there exist \textbf{\emph{tensor cores}} $A_{k}\in\R^{2\times r_{k-1}\times r_{k}}$,
$k=1,\ldots,K$, which we index as 
\[
A_{k}^{\alpha,\beta}(\sigma),\quad\alpha\in[r_{k-1}],\ \beta\in[r_{k}],\ \sigma\in\{0,1\},
\]
 such that 
\begin{equation}
T(\sigma_{1},\ldots,\sigma_{K})=\sum_{\alpha_{1}\in[r_{1}],\,\ldots,\,\alpha_{K-1}\in[r_{K-1}]}A_{1}^{1,\alpha_{1}}(\sigma_{1})A_{2}^{\alpha_{1},\alpha_{2}}(\sigma_{2})\cdots A_{d-1}^{\alpha_{d-2},\alpha_{d-1}}(\sigma_{K-1})A_{K}^{\alpha_{K-1},1}(\sigma_{K}).\label{eq:tt}
\end{equation}
 Here we have made use of the notation $[k]:=\{0,\ldots,k-1\}$, and
by convention we take $r_{0}=r_{K}=1$.

The values $r_{k}$ are called the \textbf{\emph{bond dimensions}}
or \textbf{\emph{TT ranks}}. Often we shall use the $\textsf{MATLAB}$
notation $\sigma_{k:l}=(\sigma_{k},\sigma_{k+1},\ldots,\sigma_{l})$
as a shorthand, e.g., writing tensor elements as $T(\sigma_{1:d})$.
It is also convenient to use the product shorthand $T=A_{1}A_{2}\cdots A_{K}$
for such a decomposition.

In Figure \ref{fig:diagram}, we provide a standard visual representation
of MPS/TT in tensor network diagram notation. In the current context,
where the tensor is obtained from discretization on a dyadic grid,
such a presentation of $T$ is called a \textbf{\emph{quantized tensor
train}} (QTT).

\begin{figure}
\centering{}\includegraphics[scale=0.6]{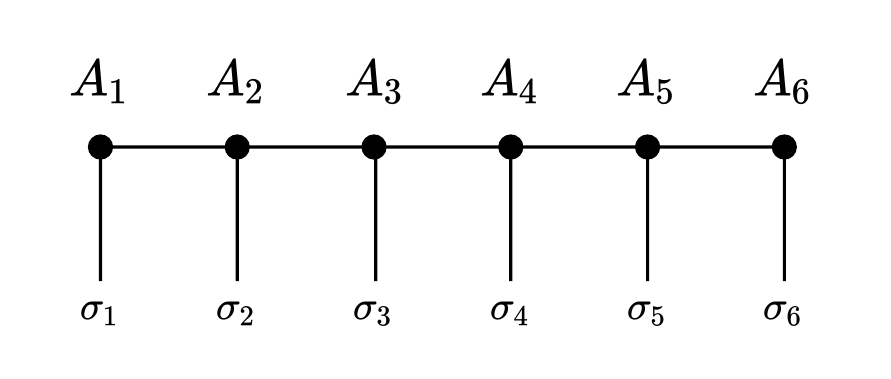}\caption{Tensor network diagram for MPS/TT (\ref{eq:tt}). The nodes indicate
tensor cores. Shared edges indicate contracted indices, while open
edges indicate indices of the tensor $T(\sigma_{1},\ldots,\sigma_{6})$
represented by the diagram.}
\label{fig:diagram}
\end{figure}

\section{Decaying rank bounds \label{sec:decaying}}

For $x\in\mathcal{D}_{K}$, which can be written uniquely as $x=\sum_{k=1}^{K}2^{-k}\sigma_{k}$,
it is useful to define the component parts 
\[
x_{\leq m}:=\sum_{k=1}^{m}2^{-k}\sigma_{k},\quad x_{>m}:=\sum_{k=m+1}^{K}2^{-k}\sigma_{k}.
\]
 Note that $x_{\leq m}$ and $x_{>m}$ are arbitrary elements of $\mathcal{D}_{m}$
and $2^{-m}\,\mathcal{D}_{K-m}$, respectively, and it is useful to
keep in mind the identifications 
\[
x_{\leq m}\leftrightarrow(\sigma_{1:m}),\quad x_{>m}\leftrightarrow(\sigma_{m+1:K}).
\]

Thinking of $v\mapsto f(u+2^{-m}v)$ as a function of $v:[0,1]\ra\R$
for fixed $u$, we can approximate it by interpolation as 
\[
f(u+2^{-m}v)\approx\sum_{\alpha}f(u+2^{-m}c^{\alpha})P^{\alpha}(v),
\]
 where $c^{\alpha}\in[0,1]$ are interpolation points (e.g., Chebyshev
nodes shifted and scaled to the interval $[0,1]$), and $P^{\alpha}$
are the corresponding interpolating functions (e.g., Chebyshev cardinal
functions, in the sense of \cite{BOYD1992}). Our error analysis will
focus on the case of Chebyshev interpolation, but we comment that
the interpolative construction that we introduce extends naturally
to other interpolation schemes. In fact, in Section \ref{sec:sparse},
we will replace ordinary Chebyshev interpolation with a notion of
local Chebyshev interpolation introduced in \cite{BOYD1992}, yielding
significant practical speedups when the number of interpolation points
is large.

Now since $f(x)=f(x_{\leq m}+x_{>m})$, we can write 
\[
T(\sigma_{1:K})=f(x)\approx\sum_{\alpha}\underbrace{f(x_{\leq m}+2^{-m}c^{\alpha})}_{=:T_{\mathrm{L}}^{\alpha}(\sigma_{1:m})}\underbrace{P^{\alpha}(2^{m}x_{>m})}_{=:T_{\mathrm{R}}^{\alpha}(\sigma_{m+1:K})},
\]
 so the $m$-th \textbf{\emph{unfolding matrix}} $T_{m}(\sigma_{1:m},\sigma_{m+1:K})$
admits a low-rank decomposition, which in turn suggests that $T$
has low TT ranks \cite{oseledets2011tensortrain}. Hence we are motivated
to bound the error of Chebyshev interpolation, which is well-understood.

\subsection{Notation}

Before stating the technical results, we fix some notation. For any
$N\geq1$, let 
\begin{equation}
c_{N}^{\alpha}=\frac{\cos(\pi\alpha/N)+1}{2},\quad\alpha=0,\ldots,N\label{eq:cheblob}
\end{equation}
 denote the \textbf{\emph{Chebyshev-Lobatto grid}}, shifted and scaled
to the interval $[0,1]$. Meanwhile let $P_{N}^{\alpha}$ denote the
Lagrange interpolating polynomials for these nodes, also known as
the Chebyshev cardinal functions, which can be evaluated directly
with simple trigonometric formulas \cite{BOYD1992}.

For measuring the error of low-rank tensor decompositions it is useful
to define the \textbf{\emph{tensor norms }}
\begin{equation}
\Vert S\Vert_{\mathrm{\infty}}=\max_{\sigma_{1:K}\in\{0,1\}^{K}}\vert S(\sigma_{1:K})\vert,\quad\Vert S\Vert_{2}=\sqrt{\frac{1}{2^{K}}\sum_{\sigma_{1:K}\in\{0,1\}^{K}}\vert S(\sigma_{1:K})\vert^{2}}\label{eq:tensornorms}
\end{equation}
 for general $S\in\R^{2}\times\cdots\times\R^{2}$ ($K$ factors).
Note that if $S$ is a quantized tensor representation for a function
$g:[0,1]\ra\R$, then our tensor 2-norm can be viewed as a Riemann
sum approximation of the $L^{2}([0,1])$ norm of $g$. Observe that
\begin{equation}
\Vert S\Vert_{2}\leq\Vert S\Vert_{\infty}\label{eq:normineq}
\end{equation}
 for all $S$.

It is useful moreover to define the ordinary \textbf{\emph{Frobenius
norm}} of a tensor $\Vert S\Vert_{\mathrm{F}}$ as the ordinary Euclidean
norm of its vectorization. Finally, we will use $\Vert\,\cdot\,\Vert$
to indicate the \textbf{\emph{operator norm}} (induced by the vector
Euclidean norm) of tensors when viewed as matrices as shall be clear
from context.

If for $p\in\{2,\infty\}$, we define the \textbf{\emph{$(\ve,p)$
rank}} of the $m$-th unfolding matrix of $T$ to be the smallest
$r$ such that 
\[
\left\Vert T-\sum_{\alpha=0}^{r-1}T_{\mathrm{L}}^{\alpha}\otimes T_{\mathrm{R}}^{\alpha}\right\Vert _{p}\leq\ve
\]
 for some tensors $T_{\mathrm{L}}^{\alpha}\in(\R^{2})^{m}$ and $T_{\mathrm{R}}^{\alpha}\in(\R^{2})^{K-m}$,
$\alpha=0,\ldots,r-1$. We will denote the $(\ve,p)$ rank of the
$m$-th unfolding matrix as $r_{m}^{(\ve,p)}[T]$. Observe that $r_{m}^{(\ve,2)}[T]\leq r_{m}^{(\ve,\infty)}[T]$.

Finally, for every level $m$, grid size $N$, and $f:[0,1]\ra\R$,
define the polynomial interpolation error 
\begin{equation}
E_{m,N}[f]:=\max_{u\in[0,1-2^{-m}],\,v\in[0,1]}\left|f(u+2^{-m}v)-\sum_{\alpha=0}^{N}f(u+2^{-m}c_{N}^{\alpha})P_{N}^{\alpha}(v)\right|.\label{eq:interperror}
\end{equation}
 Note in particular that $E_{m,N}[f]=0$ whenever $f$ is a polynomial
of degree at most $N$ because $(N+1)$-point Lagrange interpolation
is exact in this case.

\subsection{Interpolation bounds \label{sec:interp}}

First we state a lemma summarizing how the interpolation error (\ref{eq:interperror})
controls the ranks of the unfolding matrices.
\begin{lem}
Let $T$ be a quantized tensor representation for $f$ on $\mathcal{D}_{K}$.
Let $\ve>0$, and suppose that $E_{m,N}[f]\leq\ve$. Then the $(\ve,\infty)$
rank of the $m$-th unfolding matrix of $T$ is at most $N+1$.
\end{lem}

\begin{proof}
This follows directly from definitions together with the fact that
$\mathcal{D}_{m}\subset[0,1-2^{-m}]$ and $\mathcal{D}_{K-m}\subset[0,1]$.
\end{proof}
Next we show how smoothness assumptions on $f$ control the interpolation
error (\ref{eq:interperror}), which in turn, by the preceding lemma,
controls the ranks of the unfolding matrices. The first result relies
only on high-order differentiability of $f$ and is based on standard
results \cite{TrefethenBook2019} controlling Chebyshev interpolation
error.
\begin{prop}
\label{prop:diffinterp}Suppose that $f:[0,1]\ra\R$ is $p+1$ times
differentiable and that $\Vert f^{(p+1)}\Vert_{L^{\infty}([0,1])}\leq C$,
where $p\geq1$. If $N>p$, then 
\[
E_{m,N}[f]\leq\frac{4C}{\pi}\,\frac{2^{-m}}{p(N-p)^{p}}.
\]
 If $T$ is a quantized tensor representation of $f$ on $\mathcal{D}_{K}$
for $K>m$, then it follows that 
\[
r_{m}^{(\ve,\infty)}[T]\leq1+p+\left\lceil \left(\frac{4C}{\pi}\,\frac{2^{-m}}{p\ve}\right)^{1/p}\right\rceil .
\]
\end{prop}

\begin{rem}
Note that the rank bound can never drop below $p+2$, so apparently
we are penalized in our bound at large depth $m$ for using high-order
differentiability. Ultimately, we would like to claim that the rank
drops all the way to $3$ once we reach sufficient depth $m$. Indeed,
note that from the lemma it follows that if we let $C^{(q)}$ be a
pointwise bound for $\vert f^{(q+1)}\vert$ for each $q=1,\ldots p$,
then 
\[
r_{m}^{(\ve,\infty)}[T]\leq1+\left\lceil \min_{q=1,\ldots p}\left\{ q+\left[\frac{2C^{(q)}}{\pi}\,\frac{2^{-m}}{q\ve}\right]^{1/q}\right\} \right\rceil .
\]
 This rank bound satisfies $r_{m}^{(\ve,\infty)}[T]\ra3$ as $m\ra\infty$.
A more careful argument should recover $1$ as the limiting rank,
since constant interpolation is accurate at sufficiently small scales,
but we will omit more detailed statements for simplicity.\\

The proof is given in Appendix A.\\
\end{rem}

The next result is an improved bound in the case where $f$ extends
analytically to a neighborhood of the interval $[0,1]$ and is again
based on standard results \cite{TrefethenBook2019} controlling Chebyshev
interpolation error in this case. To state the result it is useful
first to make a definition: 
\begin{defn}
For $\rho>1$, define the \textbf{\emph{Bernstein ellipse }}$\mathcal{E}_{\rho}\subset\mathbb{C}$
by 
\[
\mathcal{E}_{\rho}:=\left\{ z\in\mathbb{C}\,:\,\left[\frac{\Re(z)}{a_{\rho}}\right]^{2}+\left[\frac{\Im(z)}{b_{\rho}}\right]^{2}\leq1\right\} ,
\]
 where 
\[
a_{\rho}:=\frac{\rho+\rho^{-1}}{2},\quad b_{\rho}:=\frac{\rho-\rho^{-1}}{2}.
\]

Note that $\mathcal{E}_{\rho}\supset[-1,1]$, and moreover $\mathcal{E}_{\rho'}\subset\mathcal{E}_{\rho}$
for $1<\rho'\leq\rho$.
\end{defn}

\begin{prop}
\label{prop:analyticinterp}Suppose that for some $\rho>1$, $f:[0,1]\ra\R$
extends analytically to $\mathcal{E}_{\rho}':=\frac{1}{2}\left(\mathcal{E}_{\rho}+1\right)$.
Moreover suppose that there exists $B\geq0$ such that $\vert f\vert\leq B$
on $\mathcal{E}_{\rho}'$. Let 
\[
\rho_{m}:=\max\left[\rho,\,2^{m}\frac{(\rho-1)^{2}}{\rho}\right].
\]
 Then 
\[
E_{m,N}[f]\leq\frac{4B\rho_{m}^{-N}}{\rho_{m}-1}.
\]
If $T$ is a quantized tensor representation of $f$ on $\mathcal{D}_{K}$
for $K>m$, then it follows that 
\[
r_{m}^{(\ve,\infty)}[T]\leq1+\max\left\{ 1,\,\left\lceil \log_{\rho_{m}}(1/\ve)-\log_{\rho_{m}}(\rho_{m}-1)+\log_{\rho_{m}}(4B)\right\rceil \,\right\} .
\]
\end{prop}

\begin{rem}
Consider the asymptotic rank as $m\ra\infty$, in which limit $\rho_{m}=\Omega(2^{m})$,
and think of $\ve>0$ as small. In this limit $\log_{\rho_{m}}(\rho_{m}-1)\ra1$,
$\log_{\rho_{m}}(4M)\ra0$, and $\log_{\rho_{m}}(1/\ve)\sim m^{-1}\log_{2}(1/\ve)$.
Therefore $r_{m}^{(\ve,\infty)}[T]$ is roughly bounded above by $1+\lceil m^{-1}\log_{2}(1/\ve)\rceil$.\\

The proof is also given in Appendix A.\\

Finally we prove rank bounds under the still stronger assumption that
$f$ is bandlimited.
\end{rem}

\begin{defn}
\label{def:bandlimited} We say that a function $f:\R\ra\R$ is \textbf{\emph{$\Omega$-bandlimited
}}for $\Omega>0$ if $f$ can be recovered as the inverse Fourier
transform 
\[
f(x)=\frac{1}{2\pi}\int e^{i\omega x}d\mu(\omega)
\]
 of a signed measure $\mu$ supported on the interval $[-\Omega,\Omega]$
with total variation $\vert\mu\vert<\infty$. In this case we refer
to $\mu$ as the \textbf{\emph{spectral measure }}of $f$. In particular,
if $d\mu(\omega)=\hat{f}(\omega)\,d\omega$ is absolutely continuous,
then $\vert\mu\vert=\Vert\hat{f}\Vert_{L^{1}(\R)}$, and $\hat{f}$
is supported on $[-\Omega,\Omega]$.\\
\end{defn}

\begin{prop}
\label{prop:bandlimitedinterp} Suppose that $f:[0,1]\ra\R$ is the
restriction of an $\Omega$-bandlimited function with spectral measure
$\mu$. Then 
\[
E_{m,N}[f]\le\frac{2\vert\mu\vert}{\pi}\,e^{\frac{1}{2}\left(2^{-m}\Omega-N\right)}.
\]
 If $T$ is a quantized tensor representation of $f$ on $\mathcal{D}_{K}$
for $K>m$, then it follows that 
\[
r_{m}^{(\ve,\infty)}[T]\leq1+\left\lceil 2^{-m}\Omega+2\log_{+}\left(\frac{2\vert\mu\vert}{\pi\ve}\right)\right\rceil .
\]
\end{prop}

\begin{rem}
Intuitively, the result suggests that for bandlimited functions, the
QTT ranks decay like $2^{-m}$ with the depth $m$.\\

The proof is also given in Appendix A.\\

By combining Proposition \ref{prop:bandlimitedinterp} with the trivial
bound for the leading ranks of an MPS/TT, we obtain the striking result
that the QTT ranks of an $\Omega$-bandlimited function are uniformly
bounded by $\sqrt{\Omega}+O\left(1+\log_{+}\left(\frac{\vert\mu\vert}{\ve}\right)\right)$.\\
\end{rem}

\begin{cor}
\label{cor:sqrt}Let $T$ be a quantized tensor representation for
$f$ on $\mathcal{D}_{K}$, where the depth $K$ is an arbitrary positive
integer and $f:[0,1]\ra\R$ is the restriction of an $\Omega$-bandlimited
function with spectral measure $\mu$. Then for any $\ve>0$, the
$(\ve,\infty)$ ranks of all unfolding matrices of $T$ are bounded
uniformly by $\sqrt{\Omega}+O\left(1+\log_{+}\left(\frac{\vert\mu\vert}{\ve}\right)\right)$.
\end{cor}

\begin{rem}
Consider the limit of large $\Omega$. Since the Shannon sampling
theorem suggests that we can resolve $f$ up to fixed accuracy with
a trigonometric polynomial of order $O(\Omega)$, by constructing
$f$ as a sum of $O(\Omega)$ complex exponentials, we might naively
expect that the QTT ranks of $f$ are $O(\Omega)$. However, Corollary
\ref{cor:sqrt} in fact suggests that the ranks are about $\sqrt{\Omega}$.
\end{rem}

\begin{proof}
For any $m$, the rank of the $m$-th unfolding matrix is trivially
bounded by the number of rows in this matrix, i.e., $2^{m}$. Now
for any $m\leq\lfloor\frac{1}{2}\log_{2}\Omega\rfloor$, we have in
particular that $m\leq\frac{1}{2}\log_{2}\Omega$ and therefore $2^{m}\leq\sqrt{\Omega}$.
Thus we have a uniform bound of $\sqrt{\Omega}$ on the ranks of the
leading unfolding matrices $m=1,\ldots,\lceil\frac{1}{2}\log_{2}\Omega\rceil$.

Meanwhile, for $m\geq\lceil\frac{1}{2}\log_{2}\Omega\rceil$, in particular
we have $2^{-m}\leq\frac{1}{\sqrt{\Omega}}$, so the rank bound of
Proposition \ref{prop:bandlimitedinterp} implies that 
\[
r_{m}^{(\ve,\infty)}[T]\leq1+\left\lceil \sqrt{\Omega}+2\log_{+}\left(\frac{2\vert\mu\vert}{\pi\ve}\right)\right\rceil =\sqrt{\Omega}+O\left(1+\log_{+}\left(\frac{\vert\mu\vert}{\ve}\right)\right).
\]
 Therefore we have the desired rank bound for all $m$.
\end{proof}

\section{Interpolative construction of QTTs \label{sec:interpolative}}

Although we have shown how smoothness quantitatively bounds the ranks
of the unfolding matrices of a quantized tensor representation, we
have not demonstrated an explicit construction of a QTT. In Section
\ref{sec:basicconstruction}, we will present a direct construction
based on Chebyshev interpolation using the Chebyshev-Lobatto grid
$c_{N}^{\alpha}=c^{\alpha}$, where we typically omit $N$ from this
notation going forward for visual clarity. The TT ranks in this construction
are all $N+1$.

In some settings, the numerical ranks of a quantized tensor representation
may be small even when a large Chebyshev-Lobatto grid is required
to fully resolve the target function. The downside of our basic construction
in this case is that if the quantized tensor representation has ranks
much smaller than $N$, then revealing this rank via post hoc MPS/TT
compression \cite{oseledets2011tensortrain} would require $O(N^{3})$
operations. (For the purpose of our big-O notation, we view the depth
$K$ of the network as a constant. If included, all of our big-O expressions
should include an additional factor of $K$.)

Therefore in Section \ref{sec:rankrevealing}, we present a rank-revealing
variant of the basic construction. If the true maximum TT rank (for
desired error tolerance) is $r$, then this algorithm only requires
$O(N^{2}r)$ operations.

Then in Section \ref{sec:sparse}, we show how the dense tensor cores
implementing Chebyshev interpolation can be replaced with sparse tensors,
following \cite{BOYD1992}, which further reduces the computational
cost to $O(Nr^{2})$.

Finally, in Section \ref{sec:apriori}, we show how decaying ranks
can be built directly into the QTT construction in the special case
of bandlimited functions. This section yields an \emph{a priori},
rather than \emph{a posteriori} (cf. Section \ref{sec:rankrevealing}),
guarantee of decaying QTT ranks, nearly matching our corresponding
guarantees for the unfolding matrices.

\subsection{Basic construction \label{sec:basicconstruction}}

In the expression (\ref{eq:tt}), fix the first tensor core $A_{1}=A_{\mathrm{L}}\in\R^{2\times1\times(N+1)}$
as 
\begin{equation}
A_{\mathrm{L}}^{1,\beta}(\sigma)=f\left(\frac{\sigma+c^{\beta}}{2}\right),\quad\beta\in[N+1],\ \sigma\in\{0,1\}.\label{eq:AL}
\end{equation}
 and then for $k=2,\ldots,K-1$, fix the $k$-th tensor core $A_{k}=A\in\R^{2\times(N+1)\times(N+1)}$
as 
\begin{equation}
A^{\alpha\beta}(\sigma):=P^{\alpha}\left(\frac{\sigma+c^{\beta}}{2}\right),\quad\alpha,\beta\in[N+1],\ \sigma\in\{0,1\}.\label{eq:Aint}
\end{equation}
 
\begin{figure}
\centering{}\includegraphics[scale=0.4]{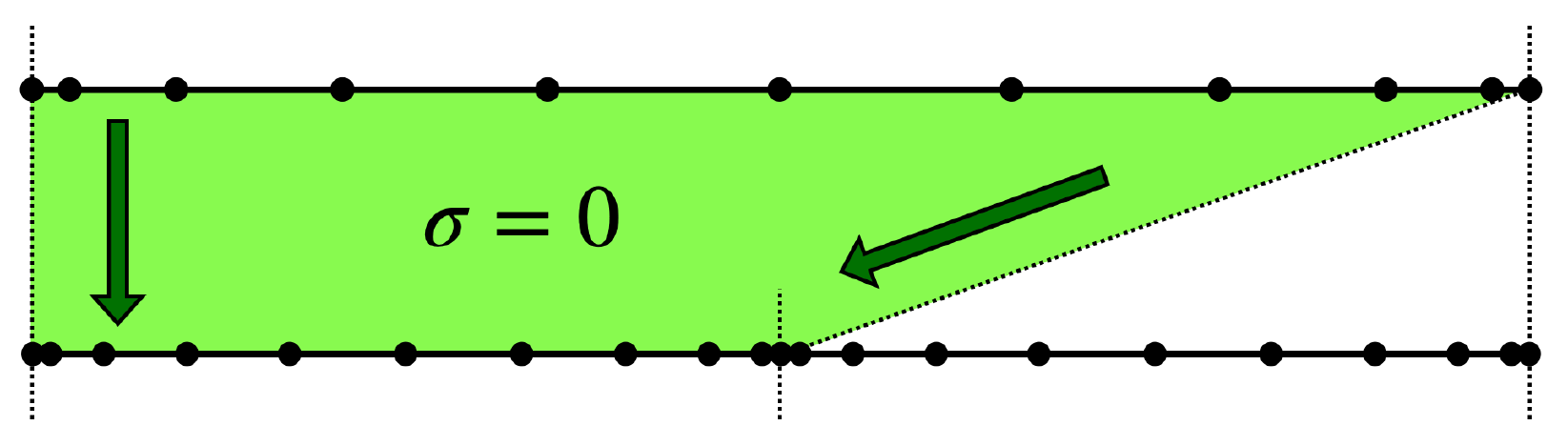}\caption{The core $A$ interpolates values from the Chebyshev-Lobatto grid
$\{x_{\protect\leq m}+2^{-m}c^{\alpha}\}_{\alpha=0}^{N}$ for the
dyadic subinterval $[x_{\protect\leq m},x_{\protect\leq m}+2^{-m}]$,
depicted on top, to the Chebyshev-Lobatto grids $\{x_{\protect\leq m}+2^{-(m+1)}\sigma+2^{-(m+1)}c^{\beta}\}_{\beta=0}^{N}$
for the left $(\sigma=0$) and right $(\sigma=1$) halves of this
dyadic subinterval. The interpolation is indicated graphically in
the figure for the case $\sigma=0$.}
\label{fig:interpcore}
\end{figure}
 Notice that all the tensor cores $k=2,\ldots,K-1$ are exactly the
same. Finally, fix the last tensor core $A_{K}=A_{\mathrm{R}}\in\R^{2\times(N+1)\times1}$
as 
\begin{equation}
A_{\mathrm{R}}^{\alpha,1}(\sigma)=P^{\alpha}\left(\frac{\sigma}{2}\right),\quad\alpha\in[N+1],\ \sigma\in\{0,1\}.\label{eq:AR}
\end{equation}
 The interpretation of the interpolating tensor core $A$ is illustrated
in Figure \ref{fig:interpcore}.

Consider the contraction of the internal tensor core with itself:
\begin{align*}
\left[A\cdot A\right]^{\alpha\beta}(\sigma,\tau):=\sum_{\gamma=0}^{N}A^{\alpha\gamma}(\sigma)A^{\gamma\beta}(\tau) & =\sum_{\gamma=0}^{N}P^{\alpha}\left(\frac{\sigma+c^{\gamma}}{2}\right)P^{\gamma}\left(\frac{\tau+c^{\beta}}{2}\right)\\
 & =\sum_{\gamma=0}^{N}P_{\sigma}^{\alpha}\left(c^{\gamma}\right)P^{\gamma}\left(x_{\tau}^{\beta}\right),
\end{align*}
 where we have defined $P_{\sigma}^{\alpha}:=P^{\alpha}\left(\frac{\sigma+\,\cdot\,}{2}\right)$
and $x_{\tau}^{\beta}:=\frac{\tau+c_{N}^{\beta}}{2}$. The last expression
can be viewed as a Lagrange interpolation formula for the value $P_{\sigma}^{\alpha}\left(x_{\tau}^{\beta}\right)$.
However, since $P_{\sigma}^{\alpha}$ is in fact a polynomial of degree
$N$, the $(N+1)$-point Lagrange interpolation is exact, and we have
the exact identity 
\[
[A\cdot A]^{\alpha\beta}(\sigma,\tau)=P_{\sigma}^{\alpha}(x_{\tau}^{\beta})=P^{\alpha}\left(\frac{\sigma+\frac{\tau+c^{\beta}}{2}}{2}\right)=P^{\alpha}\left(\frac{\sigma}{2}+\frac{\tau}{4}+\frac{c^{\beta}}{4}\right).
\]
 
\begin{figure}
\centering{}\includegraphics[scale=0.5]{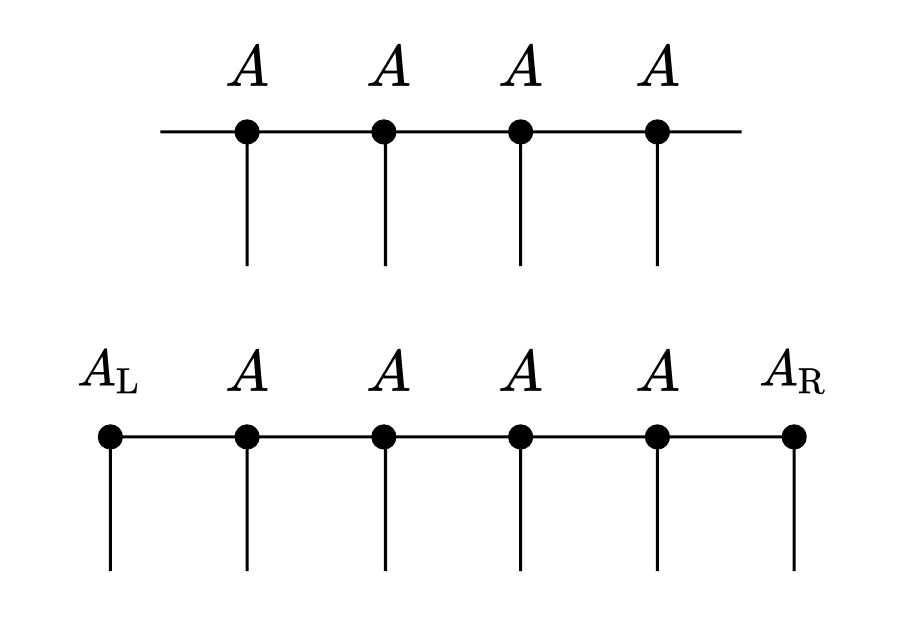}\caption{Top: tensor network diagram for the tensor $A^{p}$, where $p=4$.
Bottom: the full QTT $A_{\mathrm{L}}A^{K-2}A_{\mathrm{R}}$, where
$K=6$.}
\label{fig:interp}
\end{figure}

A straightforward inductive argument generalizes this result for arbitrary
successive contractions of $A^{p}(\sigma_{1:p})=A(\sigma_{1})A(\sigma_{2})\ldots A(\sigma_{p})$
($p$ factors) of $A$ with itself: 
\begin{lem}
\label{lem:chebcores} For any $p\geq1$, the tensor $[A^{p}]^{\alpha\beta}(\sigma_{1:p})$,
depicted graphically in Figure \ref{fig:interp}, satisfies 
\[
[A^{p}]^{\alpha\beta}(\sigma_{1:p})=P^{\alpha}\left(x_{\leq p}+2^{-p}c^{\beta}\right),
\]
 where $x_{\leq p}=\sum_{k=1}^{p}2^{-k}\sigma_{k}$. Moreover, the
tensor $\left[A^{p}A_{\mathrm{R}}\right]^{\alpha,1}(\sigma_{1:p+1})$
satisfies 
\[
[A^{p}A_{\mathrm{R}}]^{\alpha}(\sigma_{1:p+1})=P^{\alpha}\left(x_{\leq p+1}\right),
\]
 where $x_{\leq p+1}=\sum_{k=1}^{p+1}2^{-k}\sigma_{k}$.
\end{lem}

\begin{prop}
\label{prop:basicconstructionbound}Let $f:[0,1]\ra\R$, and let $T$
be its quantized tensor representation on $\mathcal{D}_{K}$. Let
$A_{\mathrm{L}}$, $A$, and $A_{\mathrm{R}}$ be tensor cores defined
as in (\ref{eq:AL}), (\ref{eq:Aint}), and (\ref{eq:AR}), respectively.
The matrix product state $S:=A_{\mathrm{L}}A^{K-2}A_{\mathrm{R}}$,
depicted graphically in Figure \ref{fig:interp}, satisfies 
\[
\Vert S-T\Vert_{\infty}\leq E_{1,N}[f].
\]
 In particular, if $f$ is a polynomial of degree at most $N$, then
$S=T$.
\end{prop}

\begin{proof}
Note that by Lemma \ref{lem:chebcores}, 
\begin{align*}
S(\sigma_{1:K}) & =\sum_{\alpha=0}^{N}f\left(\frac{\sigma_{1}+c^{\alpha}}{2}\right)P^{\alpha}\left(\sum_{k=1}^{K-1}2^{-k}\sigma_{k}\right)\\
 & =\sum_{\alpha=0}^{N}f\left(x_{\leq1}+\frac{c^{\alpha}}{2}\right)P^{\alpha}\left(2x_{>1}\right),
\end{align*}
 where $x_{\leq1}=\frac{\sigma_{1}}{2}$ and $x_{>1}=\sum_{k=2}^{K}2^{-k}\sigma_{k}$.
Hence $\vert S(\sigma_{1:K})-T(\sigma_{1:K})\vert\leq E_{1,N}[f]$
for all $\sigma_{1:K}$.
\end{proof}
Note that the construction of $S$ relies only on $2(N+1)$ evaluations
of $f$ via the construction of $A_{\mathrm{L}}$. Meanwhile $A$
and $A_{\mathrm{R}}$ are independent of $f$.

\subsection{Rank-revealing construction \label{sec:rankrevealing}}

The construction from the last section produces a QTT which exactly
matches the Chebyshev-Lobatto interpolation of $f$ on the subintervals
$[0,1/2]$ and $[1/2,1]$. In practice, the numerical TT ranks of
the tensorized representation of $f$ may be smaller than the order
$N$ used for polynomial approximation. (Indeed, we know from Propositions
\ref{prop:diffinterp} and \ref{prop:analyticinterp} that the TT
ranks of the cores should decay as we move from left to right.) In
principle, we can construct a QTT and then use standard MPS/TT compression
algorithms to reveal the true numerical TT ranks. However, the cost
of such compression is $O(N^{3})$.

In order to recover small TT ranks on the fly, we propose a rank-revealing
construction with $O(N^{2}r)$ cost, where $r$ is the true maximum
TT rank. The power of this approach will be further clarified in the
following subsection, where sparse interpolating tensors allow us
to bring the cost down further to $O(Nr^{2})$.

Suppose we have constructed $A_{\mathrm{L}}A^{p-1}$ approximately
as a matrix product state $U_{1}U_{2}\cdots U_{p}R_{p}$, depicted
graphically in Figure \ref{fig:rr}. Here the cores $U_{k}\in\R^{2\times r_{k-1}\times r_{k}}$,
indexed as $U_{k}^{\alpha\beta}(\sigma)$, have orthonormal columns
when viewed as $2r_{k-1}\times r_{k}$ matrices with respect to the
index reshaping $(\sigma\alpha,\beta)$. Meanwhile, $R_{p}\in\R^{r_{k}\times(N+1)}$
is simply a matrix.

The first decomposition $A_{\mathrm{L}}=U_{1}R_{1}$ can be obtained
by a QR decomposition of a suitable reshaping of $A_{\mathrm{L}}$.
In general, we will allow for some truncation of singular values to
reveal a bond dimension possible much smaller than $N$.

Inductively, given 
\[
A_{\mathrm{L}}A^{p-1}\approx U_{1}\cdots U_{p}R_{p},
\]
 we will obtain $U_{p+1}$ and $R_{p+1}$ such that $A_{\mathrm{L}}A^{p}\approx U_{1}\cdots U_{p+1}R_{p+1}$
as follows. First perform the contraction 
\begin{equation}
B_{p}=R_{p}A\in\R^{2\times r_{k}\times(N+1)},\quad B_{p}^{\alpha\beta}(\sigma)=\sum_{\gamma=0}^{r_{k}-1}R_{p}^{\alpha\gamma}A^{\gamma\beta}(\sigma)\label{eq:Bp}
\end{equation}
 to define a new tensor $B_{p}$. Then we can view $B_{p}$ as a $2r_{k}\times(N+1)$
matrix and perform a truncated SVD with truncation rank $r_{k+1}$
\[
B_{p}\approx U_{p+1}\Sigma_{p+1}V_{p+1}^{\top},\quad U_{p+1}\in\R^{2r_{k}\times r_{k+1}},\quad\Sigma_{p+1}\in\R^{r_{k+1}\times r_{k+1}},\quad V_{p+1}\in\R^{(N+1)\times r_{k+1}}.
\]
 Then $U_{p+1}$, viewed as a $2\times r_{k-1}\times r_{k}$ tensor,
defines our next core, and the contraction $R_{p+1}=\Sigma_{p+1}V_{p+1}^{\top}$,
which is an $r_{k+1}\times(N+1)$ matrix, completes the inductive
construction.

\begin{figure}
\centering{}\includegraphics[scale=0.5]{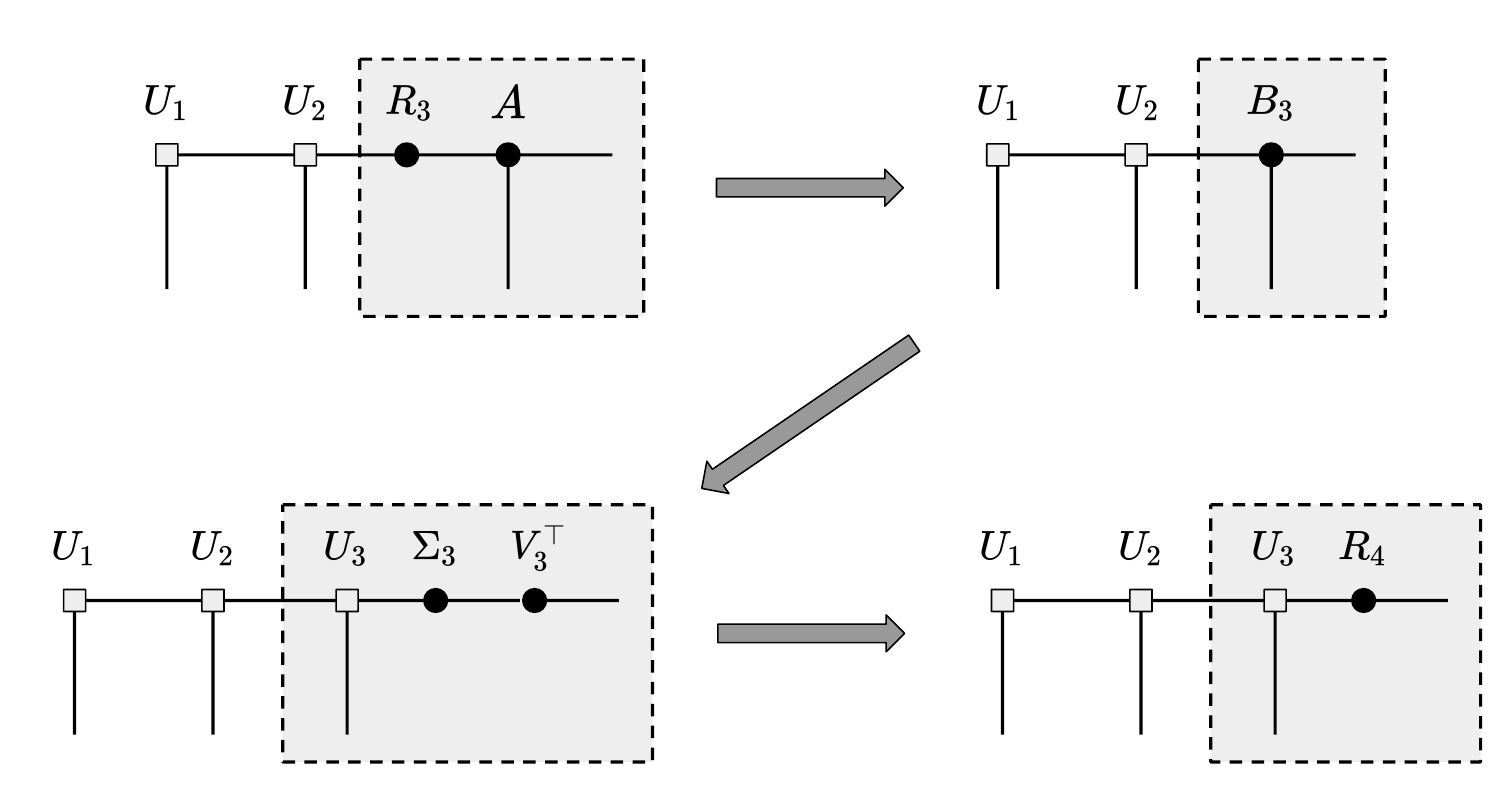}\caption{Illustration of the rank-revealing interpolative construction of QTT.}
\label{fig:rr}
\end{figure}

The steps of the construction for each level are depicted in Figure
\ref{fig:rr}. In the last level (not depicted), we simply merge the
outpute $R_{K}$ of the preceding step with the remaining tensor core
$A_{\mathrm{R}}$, without any further truncation. For every $m=1,\ldots,K-1$,
let $S_{\leq m}$ denote the tensor at the $m$-th stage of the MPS
construction. Hence $S_{\leq m}$ can be viewed as an element of $\R^{2^{m}\times(N+1)}$.
In Figure \ref{fig:rr}, for example, the last diagram depicts $S_{\leq3}$.

The following result bounds the error incurred by the SVD truncations.
\begin{thm}
\label{thm:rr} Let $\ve>0$. In the above construction, let each
SVD truncation rank $r_{k}$ be chosen as small as possible such that
the Frobenius norm error of the truncation is at most $\ve\sqrt{2^{k}}$.
Let $S$ denote the MPS that is furnished by the construction, and
let $T$ denote the true quantized tensor representation of the target
function $f$. Then the total error is bounded as 
\[
\Vert T-S\Vert_{2}\leq E_{1,N}[f]+(K-2)\,\Lambda_{N}\,\ve,
\]
 where $\Lambda_{N}$ is the Lebesgue constant \cite{TrefethenBook2019}
of $(N+1)$-point Chebyshev-Lobatto interpolation, which is in particular
bounded by 
\[
\Lambda_{N}\leq1+\frac{2}{\pi}\log(N+1).
\]
\end{thm}

\begin{rem}
From the point of view of MPS/TT compression it is surprising that
we can compress our tensor network at every stage of the construction
because the tail $A^{p}A_{\mathrm{R}}$ of tensor cores yet to be
added is not `isometric,' i.e., may amplify the compression error
that we make in the leading tensor indices. However, the fact that
this tail of tensor cores implements Chebyshev interpolation still
allows us to control the error amplification.
\end{rem}

\begin{proof}
We want to bound the error incurred by the successive SVD truncations.
We will show that at every stage of the construction $S_{\leq m}A^{K-m-1}A_{\mathrm{R}}$
remains close to the target $T$. Accordingly, let $\ve_{m}$ denote
the approximation error 
\[
\ve_{m}:=\Vert S_{\leq m}A^{K-m-1}A_{\mathrm{R}}-T\Vert_{\mathrm{F}}
\]
 for $m=1,\ldots,K-1$. Note that $S=S_{\leq K-1}A_{\mathrm{R}}$
is our final QTT, so $\ve_{K-1}$ is the total Frobenius norm error
of the construction. Meanwhile $S_{\leq1}=A_{\mathrm{L}}$, so $\ve_{1}$
is the Frobenius norm error of the uncompressed construction $A_{\mathrm{L}}A^{K-2}A_{\mathrm{R}}$,
which is bounded by Proposition \ref{prop:basicconstructionbound}
as 
\begin{equation}
\ve_{1}\leq E_{1,N}[f]\,\sqrt{2^{K}}.\label{eq:eps1}
\end{equation}

Now to bound $\ve_{m}$ inductively, compute: 
\begin{align*}
\ve_{m+1} & =\Vert T-S_{\leq m+1}A^{K-m-2}A_{\mathrm{R}}\Vert_{\mathrm{F}}\\
 & \leq\Vert T-S_{\leq m}A^{K-m-1}A_{\mathrm{R}}\Vert_{\mathrm{F}}+\Vert S_{\leq m+1}A^{K-m-2}A_{\mathrm{R}}-S_{\leq m}A^{K-m-1}A_{\mathrm{R}}\Vert_{\mathrm{F}}\\
 & =\ve_{m}+\Vert(S_{\leq m+1}-S_{\leq m}A)A^{K-m-2}A_{\mathrm{R}}\Vert_{\mathrm{F}}\\
 & \leq\ve_{m}+\Vert S_{\leq m+1}-S_{\leq m}A\Vert_{\mathrm{F}}\,\Vert A^{K-m-2}A_{\mathrm{R}}\Vert_{\mathrm{F}},
\end{align*}
 where we have used submultiplicativity of the Frobenius norm, applied
to the product $(S_{\leq m+1}-S_{\leq m}A)A^{K-m-2}A_{\mathrm{R}}$,
viewed suitably as a matrix-matrix product. Note that $\Vert S_{\leq m+1}-S_{\leq m}A\Vert_{\mathrm{F}}\leq\ve\sqrt{2^{m+1}}$,
precisely by construction. Therefore 
\begin{equation}
\ve_{m+1}\leq\ve_{m}+\ve\sqrt{2^{m+1}}\,\Vert A^{K-m-2}A_{\mathrm{R}}\Vert_{\mathrm{F}},\label{eq:epsinductive}
\end{equation}
 and it remains to bound $\Vert A^{K-m-2}A_{\mathrm{R}}\Vert_{\mathrm{F}}$. 

We introduce the shorthand notation $B=A^{K-m-2}A_{\mathrm{R}}$ and
view 
\[
B^{\alpha}(\sigma_{m+2:K}):=[A^{K-m-2}A_{\mathrm{R}}]^{\alpha}(\sigma_{m+2:K})
\]
 as a $2^{K-m-1}\times(N+1)$ matrix. By Lemma \ref{lem:chebcores},
we have that 
\[
B^{\alpha}(\sigma_{m+2:K})=P^{\alpha}\left(2^{m+1}\sum_{k=m+2}^{K}2^{-k}\sigma_{k}\right).
\]
 For further shorthand, we can write $\bsig=\sigma_{m+2:K}\in\{0,1\}^{K-m-1}$
and write $B^{\alpha}(\bsig)=P^{\alpha}(x(\bsig))$, where $x(\bsig)\in[0,1]$
is suitably defined as a function of $\bsig\in\{0,1\}^{K-m-1}$.

Then 
\begin{align*}
\Vert B\Vert_{\mathrm{F}}^{2} & =\sum_{\bsig\in\{0,1\}^{K-m-1}}\sum_{\alpha=0}^{N+1}\vert B^{\alpha}(\bsig)\vert^{2}\\
 & \leq\sum_{\bsig\in\{0,1\}^{K-m-1}}\left(\sum_{\alpha=0}^{N+1}\vert B^{\alpha}(\bsig)\vert\right)^{2}\\
 & \leq2^{K-m-1}\left(\sup_{x\in[0,1]}\sum_{\alpha=0}^{N+1}\vert P^{\alpha}(x)\vert\right)^{2}.
\end{align*}
 In fact, the quantity 
\[
\Lambda_{N}:=\sup_{x\in[0,1]}\sum_{\alpha=0}^{N+1}\vert P^{\alpha}(x)\vert
\]
 is called the \textbf{\emph{Lebesgue constant }}of the polynomial
interpolation scheme, and for the Chebyshev-Lobatto points it is known
\cite{TrefethenBook2019} that 
\[
\Lambda_{N}\leq1+\frac{2}{\pi}\log(N+1).
\]
 In summary, 
\[
\Vert A^{K-m-2}A_{\mathrm{R}}\Vert_{\mathrm{F}}=\Vert B\Vert_{\mathrm{F}}\leq\Lambda_{N}\sqrt{2^{K-m-1}}.
\]

Combining with (\ref{eq:epsinductive}), it follows that 
\[
\ve_{m+1}\leq\ve_{m}+\Lambda_{N}\,\ve\sqrt{2^{K}}.
\]
 Then from our base bound (\ref{eq:eps1}) of $\ve_{1}$, we deduce
that 
\[
\ve_{K-1}\leq\sqrt{2^{K}}\left[E_{1,N}[f]+(K-2)\,\Lambda_{N}\,\ve\right].
\]
 Now $\Vert T-S\Vert_{2}=\sqrt{2^{-K}}\,\Vert T-S\Vert_{\mathrm{F}}=\sqrt{2^{-K}}\,\ve_{K-1}$,
so we conclude that 
\[
\Vert T-S\Vert_{2}\leq E_{1,N}[f]+(K-2)\,\Lambda_{N}\,\ve,
\]
 as was to be shown.
\end{proof}

\subsection{Sparse interpolative construction \label{sec:sparse}}

As stated above our goal here is to improve the runtime of the interpolative
construction to $O(Nr^{2})$. Observe that the $\Theta(N^{2}r)$ bottleneck
in the rank-revealing construction is the construction of each core
$B_{p}$ following (\ref{eq:Bp}), which requires us to perform $2(N+1)^{2}$
sums involving up to $r$ terms each, i.e., consumes $\Theta(N^{2}r)$
runtime. If the matrices $A(\sigma)$ were \emph{sparse} with $O(1)$
nonzero entries per column, then the runtime would drop to $O(Nr^{2})$
as desired.

In fact, it is possible to construct sparse approximate Chebyshev
interpolation matrices, up to a high order of accuracy. Our construction
follows \cite{BOYD1992}, and we will review the details of the construction.

Consider a function $g:[0,1]\ra\R$. For fixed $x$, we will approximate
$g(x)$ using local Lagrange interpolation on nearby Chebyshev-Lobatto
nodes. As the number of local nodes used is increased, we converge
stably to full Chebyshev interpolation. However, if we fix the number
of local nodes and increase the underlying grid size $N$, we can
obtain rapid convergence as $N$ is increased while requiring only
sparse interpolation matrices.

It is more natural and effective to perform local Lagrange interpolations
with respect to the angular coordinate $\theta\in[0,\pi]$, related
to $x$ via $x=x(\theta)=\frac{\cos(\theta)+1}{2}$. The inverse map
is denoted $\theta=\theta(x)$. Under this correspondence, the Chebyshev-Lobatto
grid $c^{\alpha}$ is identified with an equispaced angular grid $\theta^{\alpha}=\frac{\alpha}{N}\pi$,
$\alpha=0,\ldots,N$. To perform interpolation for a function $h(\cos(\theta))$,
defined for $\theta\in[0,\pi]$ in a way that avoids boundary effects,
it is useful to extend the function to the domain $\theta\in[-\pi,2\pi]$.
Then one considers an extended angular grid $\theta^{\alpha}=\frac{\alpha}{N}\pi$,
$\alpha=-N,\ldots,2N$. For $\gamma\in\{0,\ldots,N\}$, this extension
yields the identifications $-\gamma\sim\gamma$ and $N+\gamma\sim N-\gamma$.
For any $\alpha\in\{-N,\ldots,2N\}$, we let $\llbracket\alpha\rrbracket$
denote the unique representative of $\alpha$ in $\{0,\ldots,N\}$
up to this equivalence.

Now for every $\theta\in[0,\pi]$, let $\iota(\theta)$ denote the
index of the closest angular grid point in $\{\theta^{\alpha}\}$.
We let $M\leq N$ denote a hyperparameter which determines the order
of the local Lagrange interpolation. Then we approximate $g(x)$ with
its local Lagrange interpolation in the angular coordinate using interpolation
points $\theta^{\gamma}$, where $\gamma=\iota(\theta)-M,\ldots,\iota(\theta)+M$.
Concretely, we approximate: 
\begin{equation}
g(x)\approx\sum_{\gamma=\iota(\theta(x))-M}^{\iota(\theta(x))+M}g(c^{\llbracket\gamma\rrbracket})\,L^{\gamma}(\theta(x)),\label{eq:locallagrange}
\end{equation}
 where 
\[
L^{\gamma}(\theta)=\prod_{\beta\in\{-M,\ldots,M\}\backslash\{\gamma\}}\frac{\theta-\theta^{\beta}}{\theta^{\gamma}-\theta^{\beta}}
\]
 are the Lagrange basis functions.

We can think of the right-hand side of (\ref{eq:locallagrange}) as
defining a linear operator $\mathcal{I}=\mathcal{I}_{M,N}$ on the
space of functions $[0,1]\ra\R$ which sends $g$ to its interpolation
$\mathcal{I}g$.

If we assume that $g$ can be approximated by an ordinary Chebyshev
interpolation as 
\[
g\approx\sum_{\alpha=0}^{N}g(c^{\alpha})\,P^{\alpha},
\]
 then 
\begin{equation}
g\approx\mathcal{I}g\approx\sum_{\alpha=0}^{N}g(c^{\alpha})\,\mathcal{I}P^{\alpha},\label{eq:boydinterp}
\end{equation}
 so we are motivated to compute 
\begin{equation}
\mathcal{I}P^{\alpha}(x)=\sum_{\gamma=\iota(\theta(x))-M}^{\iota(\theta(x))+M}\delta^{\alpha,\llbracket\gamma\rrbracket}\,L^{\gamma}(\theta(x)),\label{eq:IPA}
\end{equation}
 where we have used the fact that $P^{\alpha}(c^{\beta})=\delta^{\alpha,\beta}$.

Recall moreover that in the basic construction of Section \ref{sec:basicconstruction},
we were interested in using evaluations $g(c^{\alpha})$ to compute
the values $g\left(\frac{\sigma+c^{\beta}}{2}\right)$, for $\sigma\in\{0,1\}$,
which we achieved via $g\left(\frac{\sigma+c^{\beta}}{2}\right)=\sum_{\alpha=0}^{N}g(c^{\alpha})A^{\alpha\beta}(\sigma)$.
But from (\ref{eq:boydinterp}), we have that 
\[
g\left(\frac{\sigma+c^{\beta}}{2}\right)\approx\sum_{\alpha=0}^{N}g(c^{\alpha})\,\tilde{A}^{\alpha\beta}(\sigma),
\]
 where the tensor core $\tilde{A}\in\R^{2\times(N+1)\times(N+1)}$
is defined by 
\begin{equation}
\tilde{A}^{\alpha\beta}(\sigma)=\mathcal{I}P^{\alpha}\left(\frac{\sigma+c^{\beta}}{2}\right),\label{eq:Atilde}
\end{equation}
 where $\mathcal{I}P^{\alpha}$ is defined as in (\ref{eq:IPA}).

Observe that $\tilde{A}^{\alpha\beta}(\sigma)$ is nonzero only if
\[
\left|\pi^{-1}\theta\left(\frac{\sigma+c^{\beta}}{2}\right)-\alpha\right|\leq M+1.
\]
 It follows that each column of $\tilde{A}(\sigma)$ has $O(M)$ nonzero
entries as $N\ra\infty$.

Then our sparse interpolative construction is achieved by replacing
the tensor core $A$ defined in Section \ref{sec:basicconstruction}
with $\tilde{A}$ defined by (\ref{eq:Atilde}), which depends on
the additional parameter $M$ controlling the order of local interpolation.
Then the rank-revealing construction of Section \ref{sec:rankrevealing}
is applied with $\tilde{A}$ in the place of $A$.

\subsection{\emph{A priori }decaying-rank construction \label{sec:apriori}}

Recall that Proposition \ref{prop:basicconstructionbound} guarantees
that our basic construction of Section \ref{sec:basicconstruction}
is accurate. Note that in this construction, all of the QTT ranks
are of the same size (i.e., $N+1$). Meanwhile, Theorem \ref{thm:rr}
guarantees that truncating the ranks `on-the-fly' to specified tolerance
is safe, in the sense that the error of the full QTT is controlled
by this tolerance. However, this \emph{a posteriori }bound does not
directly imply an \emph{a priori }guarantee that the decaying ranks
of the unfolding matrices, proved in Section \ref{sec:decaying},
will be revealed.

Ideally, we would like an \emph{a priori }construction with decaying
ranks, matching the rank guarantees for the unfolding matrices. For
simplicity we will focus on the bandlimited case, where we have an
intuition that the ideal ranks should decrease by roughly a factor
of $2$ at each level of the QTT. Therefore we should consider interpolating
grids that decay in size by roughly a factor of $2$ at each level.

Concretely, consider $f$ which is $\Omega$-bandlimited in the sense
of Definition \ref{def:bandlimited}. Then define grid sizes 
\[
N_{k}=\left\lceil 2^{-k}\Omega+\Delta\right\rceil ,\quad k=1,\ldots,K,
\]
 where $\Delta>0$ is a free parameter, the choice of which we will
discuss later.

Then we define interpolating cores $A_{k}\in\R^{2\times N_{k-1}\times N_{k}}$
for $k=2,\ldots,K$: 
\begin{equation}
A_{k}^{\alpha\beta}(\sigma):=P_{N_{k-1}}^{\alpha}\left(\frac{\sigma+c_{N_{k}}^{\beta}}{2}\right),\quad\alpha\in[N_{k-1}+1],\ \beta\in[N_{k}+1],\ \sigma\in\{0,1\}.\label{eq:Aint-1}
\end{equation}
 Following the shorthand adopted in preceding sections, our QTT will
be defined in terms of its cores as 
\begin{equation}
S=A_{\mathrm{L}}A_{2},\ldots,A_{K-1}A_{\mathrm{R}},\label{eq:Sdecay}
\end{equation}
 where $A_{\mathrm{L}}$ is defined as in (\ref{eq:AL}) with $N=N_{1}$
and $A_{\mathrm{R}}$ is defined by $A_{\mathrm{R}}^{\alpha}(\sigma):=A_{K}^{\alpha0}(\sigma)$.
\begin{thm}
Let $S$ denote the MPS (\ref{eq:Sdecay}) that is furnished by the
above construction, and let $T$ denote the true quantized tensor
representation of the target function $f:[0,1]\ra\R$, which we assume
to be the restriction of an $\Omega$-bandlimited function with spectral
measure $\mu$. The construction error is bounded as 
\[
\Vert T-S\Vert_{\infty}\leq\frac{2}{\pi}\vert\mu\vert(K-1)\left[1+\frac{2}{\pi}\log(\lceil\Omega/2+\Delta\rceil+1)\right]^{K-2}e^{-\Delta/2}.
\]
\end{thm}

\begin{rem}
The factor in the bound with the $K$-dependent exponent results from
repeated error amplification by the Lebesgue constant of the interpolating
cores. In practice, this aspect of the bound seems not to be sharp,
though it appears to be difficult to rule out such a dependence. Nonetheless,
this factor can be suppressed efficiently in the bound by only a modest
increase of $\Delta$ with mild dependence on the depth $K$.
\end{rem}

\begin{rem}
Note, following the same reasoning as in the proof of Corollary \ref{cor:sqrt},
that the largest QTT rank of this construction is bounded by $\lceil\sqrt{\Omega}+\Delta\rceil+1$,
since the leading ranks $k=1,\ldots,\lfloor\frac{1}{2}\log_{2}\Omega\rfloor$
are trivially bounded by $\sqrt{\Omega}$ (the number of rows in the
unfolding matrix) and the ranks $k\geq\lceil\frac{1}{2}\log_{2}\Omega\rceil$
are bounded by $\lceil\sqrt{\Omega}+\Delta\rceil+1$.
\end{rem}

\begin{proof}
As in Section \ref{sec:rankrevealing}, let $S_{\leq m}=A_{\mathrm{L}}A_{2}\cdots A_{m}$
denote our QTT at the $m$-th stage of construction. Let 
\[
\ve_{m}:=\Vert S_{\leq m}-T_{\leq m}\Vert_{\infty},
\]
 where $T_{\leq m}$ is defined by 
\[
T_{\leq m}^{\alpha}(\sigma_{1:m})=f\left(\sum_{k=1}^{m}2^{-k}\sigma_{k}+2^{-m}c_{N_{m}}^{\alpha}\right).
\]
 Note that then in particular $\Vert T-S\Vert_{\infty}\leq\ve_{K}$,
so we seek a bound on $\ve_{K}$, which we build inductively.

Compute: 
\begin{align*}
\ve_{m+1} & =\Vert S_{\leq m}A_{m+1}-T_{\leq m+1}\Vert_{\infty}\\
 & \leq\Vert(S_{\leq m}-T_{\leq m})A_{m+1}\Vert_{\infty}+\Vert T_{\leq m}A_{m+1}-T_{\leq m+1}\Vert_{\infty}\\
 & \leq\Lambda_{N_{m}}\underbrace{\Vert S_{\leq m}-T_{\leq m}\Vert_{\infty}}_{=\,\ve_{m}}+E_{m,N_{m}}[f],
\end{align*}
 where as in the proof of Theorem \ref{thm:rr}, $\Lambda_{N_{m}}$
denotes the Lebesgue constant for $(N_{m}+1)$-point Chebyshev-Lobatto
interpolation. To proceed from the penultimate line to the last line,
we used (1) the fact that the Lebesgue constant bounds the $\infty$-operator
norm of polynomial interpolation and (2) the fact that the error $T_{\leq m}A_{m}-T_{\leq m+1}$
is a Chebyshev interpolation error precisely in the sense of (\ref{eq:interperror}).

It follows then, via Proposition \ref{prop:bandlimitedinterp}, that
\[
\ve_{m+1}\leq\Lambda_{N_{m}}\ve_{m}+\frac{2\vert\mu\vert}{\pi}e^{\frac{1}{2}(2^{-m}\Omega-N_{m})}.
\]
 Since $N_{m}\ge2^{-m}\Omega+\Delta$ by definition, we have 
\[
\ve_{m+1}\leq\Lambda_{N_{m}}\ve_{m}+\frac{2\vert\mu\vert}{\pi}e^{-\frac{\Delta}{2}}.
\]
 Then recalling from \cite{TrefethenBook2019} that 
\[
\Lambda_{M}\leq1+\frac{2}{\pi}\log(M+1)
\]
 for all $M$, we have 
\begin{equation}
\ve_{m+1}\leq a\ve_{m}+b,\label{eq:epsrecurse}
\end{equation}
 where $a:=1+\frac{2}{\pi}\log(\lceil\Omega/2+\Delta\rceil+1)$ and
$b:=\frac{2\vert\mu\vert}{\pi}e^{-\Delta/2}$.

From (\ref{eq:epsrecurse}), together with the base case $\ve_{1}=0$,
it follows that 
\[
\ve_{K}\leq b\sum_{k=0}^{K-2}a^{k}\leq b(K-1)a^{K-2},
\]
 i.e., 
\[
\ve_{K}\leq\frac{2}{\pi}\vert\mu\vert(K-1)\left[1+\frac{2}{\pi}\log(\lceil\Omega/2+\Delta\rceil+1)\right]^{K-2}e^{-\Delta/2}.
\]
 Since $\Vert T-S\Vert_{\infty}\leq\ve_{K}$, this implies the desired
bound on $\Vert T-S\Vert_{\infty}$.
\end{proof}

\section{Inverting the construction \label{sec:invert}}

Note that the tensor $S_{\leq m}$ at stage $m$ of our construction
(whether or not sparse interpolation and/or SVD truncation are applied)
is an approximation of the `ground truth' tensor $T_{\leq m}$ defined
by 
\[
T_{\leq m}^{\beta}(\sigma_{1:m})=f\left(x_{\leq m}+2^{-m}c^{\beta}\right),
\]
 where $x_{\leq m}=\sum_{k=1}^{m}2^{-k}\sigma_{k}$, consisting of
evaluations of $f$ on Chebyshev-Lobatto grids, shifted and scaled
to dyadic subintervals of $[0,1]$.

It is natural to ask whether it is possible to `invert' our construction
to recover such evaluations from a given QTT for a function $f$,
which only directly furnishes evaluations of $f$ on the dyadic grid
$\mathcal{D}_{K}$. This can be achived in two stages:
\begin{enumerate}
\item Recover evaluations of $f$ on small-scale Chebyshev-Lobatto grids
using Lagrange interpolation on small-scale equispaced grids.
\item Recover $f$ on larger scale Chebyshev-Lobatto grids from finer Chebyshev-Lobatto
grids by Chebyshev interpolation.
\end{enumerate}
To accomplish the first stage, consider the Lagrange interpolating
polynomials on a small-scale dyadic grid $\mathcal{D}_{q}\subset[0,1]$,
which we evaluate on the Chebyshev-Lobatto grid $\{c^{\beta}\}\subset[0,1]$.
The dyadic grid points in $\mathcal{D}_{q}$ can be written $\sum_{k=1}^{q}2^{-k}\sigma_{k}$,
indexed by $\sigma_{1:q}\in\{0,1\}^{q}$, which motivates us to define
the Lagrange interpolation tensor $L\in\R^{2^{q}\times(N+1)}$ 
\[
L^{\beta}(\sigma_{1:q})=\prod_{\tau_{1:q}\in\{0,1\}^{q}\backslash\{\sigma_{1:q}\}}\frac{c^{\beta}-\sum_{k=1}^{q}2^{-k}\tau_{k}}{\sum_{k=1}^{q}2^{-k}(\sigma_{k}-\tau_{k})}.
\]
 $L$ can be contracted with $S$, as indicated at the top of Figure
\ref{fig:invert}, to obtain an approximation $S_{\leq K-q}$ for
$T_{\leq K-q}$. The expression for this contraction is written 
\[
S_{\leq K-q}^{\beta}(\sigma_{1:K-q})=\sum_{\sigma_{K-q+1:K}\in\{0,1\}^{q}}S(\sigma_{1:K})\,L^{\beta}(\sigma_{K-q+1:K}).
\]

Due to Runge's phenomenon, it is not safe to take $q$ large while
the depth $K$ is fixed. However, since the conceit of QTT is that
the depth $K$ can be taken large enough to resolve all the fine-scale
structure of the target function $f$, even Lagrange interpolation
with a fixed small value of $q$ will be very accurate. For smooth
functions, the error will be exponentially small in $K$ for fixed
$q$, with more rapid asymptotic convergence in $K$ when $q$ is
larger. For many concrete purposes, $q=1$ should suffice. We will
not make any explicit careful statement, though standard Lagrange
error bounds can be consulted.

\begin{figure}
\centering{}\includegraphics[scale=0.5]{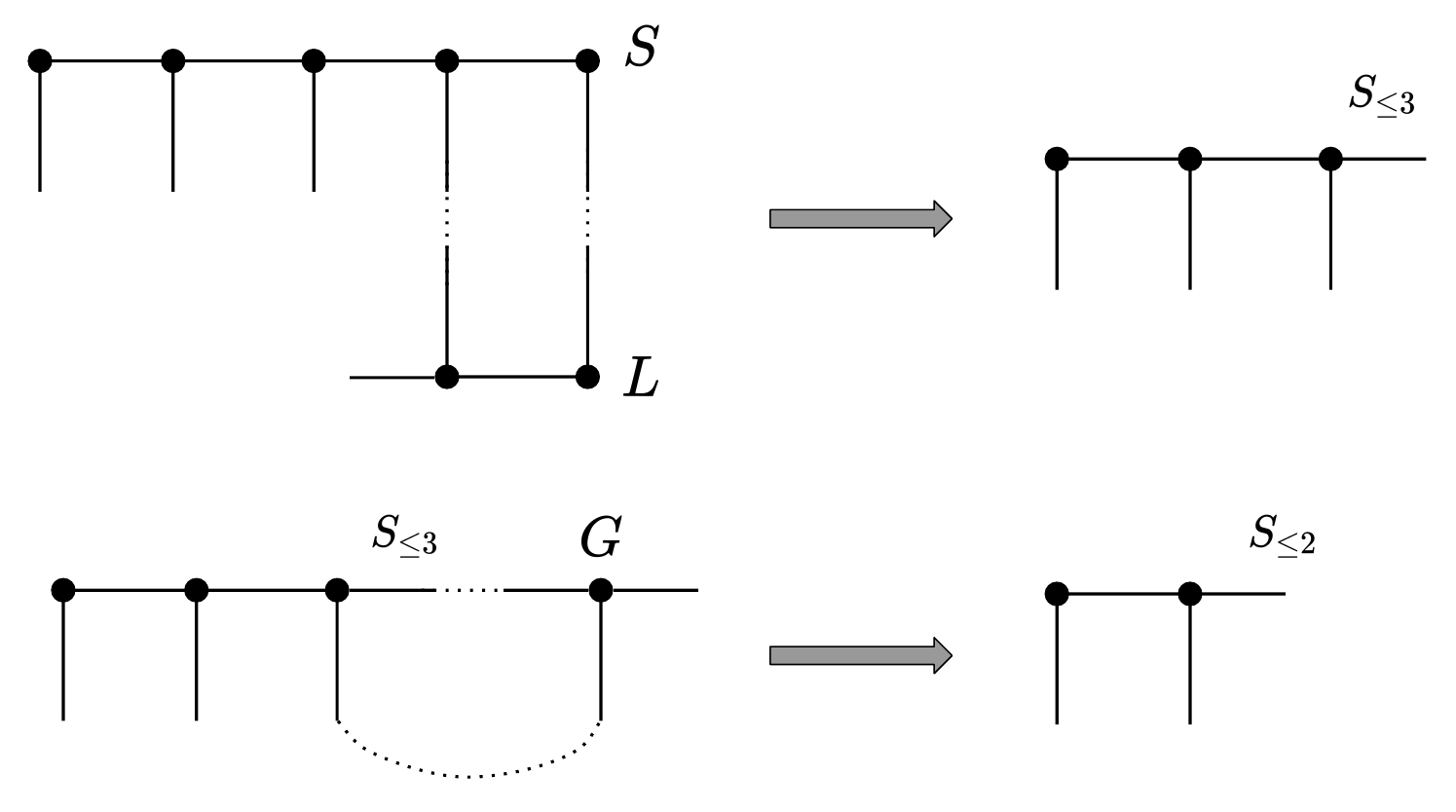}\caption{Top: stage 1 of the inversion procedure. Bottom: one level of stage
2.}
\label{fig:invert}
\end{figure}

Now we turn to stage 2 of the inversion procedure. To inductively
obtain $S_{\leq k}$ from $S_{\leq k+1}$, we want to invert the operation
of tacking on a single tensor core $A$. Accordingly, we want to define
a tensor $G\in\R^{2\times(N+1)\times(N+1)}$, indexed as $G^{\alpha\beta}(\sigma)$,
such that 
\begin{equation}
\sum_{\sigma\in\{0,1\}}\sum_{\gamma=0}^{N}A^{\alpha\beta}(\sigma)G^{\beta\gamma}(\sigma)=\delta^{\alpha\gamma},\label{eq:inversion}
\end{equation}
 i.e., we seek $G$ which is a generalized inverse of $A$, viewed
appropately as a matrix of shape $2(N+1)\times(N+1)$.

This problem is underdetermined, but there exists a solution $G$
whose entries remain bounded as the interpolation grid size $N$ becomes
large: 
\begin{equation}
G^{\alpha\beta}(\sigma)=\begin{cases}
\delta_{\sigma0}P^{\alpha}\left(2c^{\beta}\right), & c^{\beta}\in[0,1/2]\\
\delta_{\sigma1}P^{\alpha}(2c^{\beta}-1), & c^{\beta}\in(1/2,1].
\end{cases}\label{eq:G}
\end{equation}
 Note that by construction $\vert G^{\alpha\beta}(\sigma)\vert\leq1$
for all $\alpha,\beta,\sigma$.
\begin{lem}
The tensor core $G$ as defined in (\ref{eq:G}) satisfies the inversion
property (\ref{eq:inversion}). Moreover, all entries of $G$ are
bounded by $1$ in absolute value.
\end{lem}

\begin{proof}
It remains only to verify the inversion property. Compute: 
\[
\sum_{\sigma\in\{0,1\}}\sum_{\beta=0}^{N}A^{\alpha\beta}(\sigma)G^{\beta\gamma}(\sigma)=\begin{cases}
\sum_{\beta=0}^{N}A^{\alpha\beta}(0)G^{\beta\gamma}(0), & c^{\gamma}\in[0,1/2],\\
\sum_{\beta=0}^{N}A^{\alpha\beta}(1)G^{\beta\gamma}(1), & c^{\gamma}\in(1/2,1].
\end{cases}
\]
 Note that for either $\sigma\in\{0,1\}$, 
\begin{align*}
\sum_{\beta=0}^{N}A^{\alpha\beta}(\sigma)G^{\beta\gamma}(\sigma) & =\sum_{\beta}P^{\alpha}\left(\frac{\sigma+c^{\beta}}{2}\right)P^{\beta}(2c^{\gamma}-\sigma)\\
 & =P^{\alpha}\left(\frac{\sigma+(2c^{\gamma}-\sigma)}{2}\right)\\
 & =P^{\alpha}(c^{\gamma})\\
 & =\delta^{\alpha\gamma},
\end{align*}
 where we have used the fact that $(N+1)$-point polynomial interpolation
of the degree-$N$ polynomial $P^{\alpha}\left(\frac{\sigma+\,\cdot\,}{2}\right)$
is exact.
\end{proof}

\section{Multiresolution interpolative construction \label{sec:sharp}}

So far our interpolative construction depends only evaluations of
the target function $f$ on a single coarse grid $\frac{\sigma+c^{\alpha}}{2}$,
$\sigma\in\{0,1\}$, $\alpha=0,\ldots,N$. In this construction we
must take $N$ large enough to resolve all the features of $f$, as
we have quantified in Proposition \ref{prop:basicconstructionbound}.
It can be observed empirically, however, that certain functions, which
may even be nonsmooth, have low QTT ranks even though they cannot
be interpolated from a single coarse grid.

In this section, we explain this behavior and provide a direct construction
that achieves low TT ranks using additional \emph{a priori} knowledge.
We comment that a rank-revealing construction using sparse interpolation,
as outlined in Sections \ref{sec:rankrevealing} and \ref{sec:sparse},
may still be adequate for the practical purpose of compressing some
target function $f$.

Suppose that at each level $k<K$, we are given a collection $\mathcal{S}_{k}$
of multi-indices $\sigma_{1:k}^{i}=(\sigma_{1}^{i},\ldots,\sigma_{k}^{i})$,
where $i=1,\ldots,q_{k}$. Concretely: 
\[
\mathcal{S}_{k}=\left\{ \sigma_{1:k}^{i}\,:\,i=1,\ldots,q_{k}\right\} .
\]
 For each $i$, let 
\[
x_{\leq k}^{i}=\sum_{l=1}^{k}2^{-l}\sigma_{l}
\]
 denote the dyadic grid point corresponding to $\sigma_{1:k}^ {}$.
By convention we take $\mathcal{S}_{K}=\emptyset$.

We will think of the dyadic subintervals $[x_{\leq k}^{i},x_{\leq k}^{i}+2^{-k}]$
as a collection of `dangerous' subintervals on which the sharp behavior
of $f$ makes it too dangerous to interpolate. Within these subintervals,
we defer function evaluation.

Importantly, we assume that each $\sigma_{1:k+1}^{i}\in\mathcal{S}_{k+1}$
can be written as $(\sigma_{1:k}^{j},\sigma)$ for some $\sigma_{1:k}^{j}\in\mathcal{S}_{k}$
and $\sigma\in\{0,1\}$. In other words, each dangerous subinterval
is contained within a dangerous subinterval at the next largest scale.

For example, consider a function such as $f(x)=\sqrt{x}$, with a
cusp at the left endpoint of the interval $[0,1]$. In this case it
will be effective to take $q_{k}=1$ and $\sigma_{1:k}^{1}=(0,0,\ldots,0)$
for each $k$. Accordingly, we view the subintervals $[0,2^{-k}]$
as dangerous, but all other dyadic subintervals, such as $[2^{-k},2^{-k}+2^{-m}]$
for $m\geq k$, are viewed as safe.

We illustrate a choice of subintervals for a function with a cusp
in Figure \ref{fig:subintervals}.

\begin{figure}
\centering{}\includegraphics[scale=0.7]{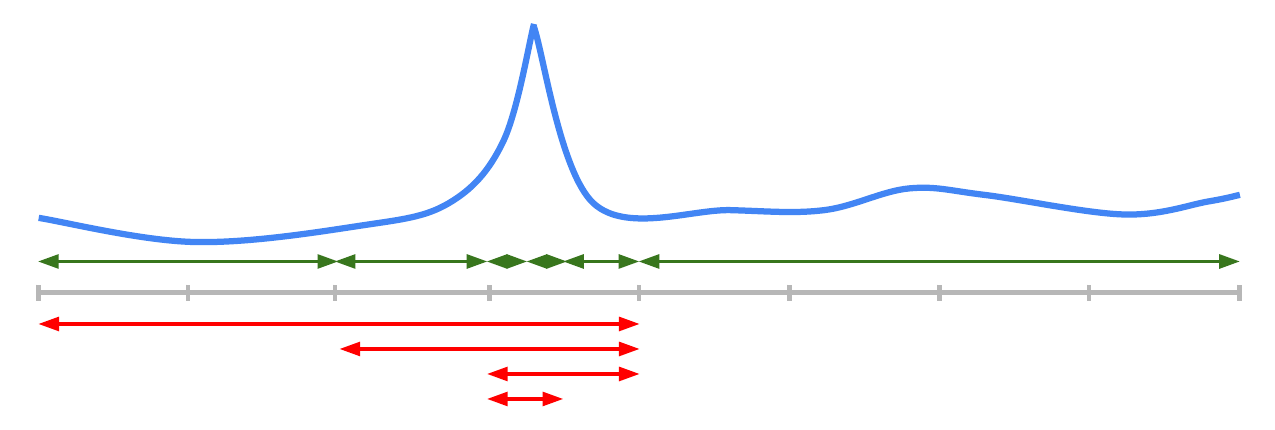}\caption{Consider a function with a single cusp, whose graph is indicated by
the blue line. The red intervals indicate our choice of `dangerous'
subintervals at each level of resolution. The resulting QTT with $K=4$
will perform exact Chebyshev-Lobatto interpolation using function
evaluations on the subintervals indicated in green.}
\label{fig:subintervals}
\end{figure}

Our inductive target for the QTT construction is the tensor $T_{\leq k}\in\R^{2^{k}\times(N+1+q_{k})}$
defined by stacking two tensors $T_{\leq k}^{\,\mathrm{up}}\in\R^{2^{k}\times(N+1)}$
and $T_{\leq k}^{\,\mathrm{down}}\in\R^{2^{k}\times q_{k}}$, where
\[
\left[T_{\leq k}^{\,\mathrm{up}}\right]^{\beta}(\sigma_{1:k})=\begin{cases}
f\left(x_{\leq k}+2^{-k}c^{\beta}\right), & \sigma_{1:k}\notin\mathcal{S}_{k},\\
0, & \text{otherwise},
\end{cases}
\]
 and 
\[
\left[T_{\leq k}^{\,\mathrm{down}}\right]^{i}(\sigma_{1:k})=\begin{cases}
1, & \sigma_{1:k}=\sigma_{1:k}^{i},\\
0, & \text{otherwise}.
\end{cases}
\]

Thus $T_{\leq k}(\sigma_{1:k})$ stores several pieces of information.
First, it tells us whether or not we are in a dangerous interval.
In this case, the top part of the vector is the zero vector, and the
bottom part is an indicator vector telling us which dangerous subinterval
we are in. Second, it tells us, if we are in a safe interval, the
function values that we will need to interpolate the function on this
interval.

We need to determine the tensor core $A_{k}$ (extending the core
$A$ introduced in Section \ref{sec:basicconstruction}), that can
be attached to $T_{\leq k}$ to obtain $T_{\leq k+1}$ approximately.
This tensor can be defined in a block sense as 
\[
A_{k}(\sigma)=\left(\begin{array}{cc}
A(\sigma) & 0\\
F_{k}(\sigma) & \chi_{k}(\sigma)
\end{array}\right).
\]
 Here the entire matrix is $(N+1+q_{k-1})\times(N+1+q_{k})$ and the
upper-left block is $(N+1)\times(N+1)$. The lower blocks are defined
as 
\[
F_{k}^{i\beta}(\sigma)=\begin{cases}
0, & (\sigma_{k-1}^{i},\sigma)\in\mathcal{S}_{k},\\
f(x_{\leq k-1}^{i}+2^{-k}\sigma+2^{-k}c^{\beta}), & \text{otherwise},
\end{cases}
\]
 and 
\[
\chi_{k}^{ij}(\sigma)=\begin{cases}
1, & \sigma_{1:k}^{j}=(\sigma_{1:k-1}^{i},\sigma),\\
0, & \text{otherwise}.
\end{cases}
\]

For the initial tensor core, we can simply take $A_{1}=T_{\leq1}$
directly. For the final tensor core, define 
\[
A_{K}(\sigma)=\left(\begin{array}{c}
A_{\mathrm{R}}(\sigma)\\
F_{K}(\sigma)
\end{array}\right),
\]
 where 
\[
F_{K}^{i,1}(\sigma)=f(x_{\leq K-1}^{i}+2^{-K}\sigma).
\]

The entire QTT is constructed as $S:=A_{1}A_{2}\cdots A_{K}$, which
is shorthand for (\ref{eq:tt}). Careful inspection reveals the following
interpretation of the entries $S(\sigma_{1:K})$: 
\begin{thm}
\label{thm:multires} The QTT $S$ constructed following the above
procedure admits the following interpretation. For fixed $\sigma_{1:K}$,
consider the smallest $k$ such that $\sigma_{1:k}\notin\mathcal{S}_{k}$.
If $k\leq K-1$, then 
\[
S(\sigma_{1:K})=\sum_{\alpha=0}^{N}f(x_{\leq k}+2^{-k}c^{\alpha})\,P^{\alpha}(2^{k}x_{>k}),
\]
 where $x_{\leq k}=\sum_{l=1}^{k}2^{-l}\sigma_{l}$ and $x_{>k}=\sum_{l=k+1}^{K}2^{-l}\sigma_{l}$.
In other words, the value $S(\sigma_{1:K})$ is furnished in this
case by interpolation of $f$, evaluated at the point $x_{\leq K}$,
using a Chebyshev-Lobatto grid shifted and scaled to the interval
$\left[x_{\leq k},x_{\leq k}+2^{-k}\right]$. Meanwhile, if $k=K$,
then $S(\sigma_{1:K})$ is furnished by the exact evaluation $f(x_{\leq K})$.

Letting $T$ denote the exact quantized tensor representation of $f$,
it follows that 
\[
\Vert S-T\Vert_{\infty}\leq\max\left\{ E_{I,N}[f]\ :I\notin\bigcup_{k=1}^{K-1}\mathcal{S}_{k}\ \text{is a dyadic subinterval of}\ [0,1]\ \text{of length at least }2^{-(K-1)}\right\} .
\]
 In the last expression, the dyadic subintervals are the intervals
of the form $[x,x+2^{-k}]$ where $x\in\mathcal{D}_{k}$ and $k\in\{1,\ldots,K-1\}$,
and $E_{I,N}[f]$ denotes the $L^{\infty}$-norm error of $(N+1)$-point
Chebyshev-Lobatto interpolation of $f$ on the interval $I$.
\end{thm}

Using this error bound, it is simple to derive \emph{a priori }bounds
for the compression of certain functions. For example, consider the
function $f(x)=\sqrt{x}$ with the subinterval selection indicated
above---i.e., $q_{k}=1$ and $\sigma_{1:k}^{1}=(0,0,\ldots,0)$ for
each $k=1,\ldots,K-1$. The `worst' subintervals where we need to
control the interpolation are the subintervals $[2^{-k},2^{1-k}]$,
$k=1,\ldots,K-1$. In fact, $f$ is self-similar on this collection
of intervals, and the worst of these is the case $k=1$ (since the
restrictions of $f$ to all the other intervals, rescaled to the same
domain $[1/2,1]$ are scalar multiples with decaying prefactors).
Hence we are motivated to bound $(N+1)$-point Chebyshev-Lobatto interpolation
error of $f(x)=\sqrt{x}$ on the interval $[1/2,1]$. In turn we are
motivated to bound the interpolation error of $g(x)=\frac{1}{2}\sqrt{x+3}$
on the reference interval $[-1,1]$. Now $g$ extends analytically
to the Bernstein ellipse $\mathcal{E}_{\rho}$ for $\rho\in[1,3+2\sqrt{2}]$
and is bounded by $M_{\rho}=\frac{1}{2}\sqrt{\frac{\rho+\rho^{-1}}{2}+3}$
on this region. Consider the extreme choice $\rho=3+2\sqrt{2}$, yielding
$M_{\rho}=\frac{\sqrt{6}}{2}$, so by applying Theorem 8.2 of \cite{TrefethenBook2019}
(cf. the proof Proposition \ref{prop:analyticinterp}) we have that
the interpolation error is bounded by 
\[
\frac{2\sqrt{6}\rho^{-N}}{\rho-1}\leq(1.015)\times(5.828)^{-N}.
\]

\section{Multivariate functions \label{sec:multivariate}}

Several approaches have been considered for extending the use of quantized
tensor trains to multivariate functions; see for instance \cite{ye2022,ye2023quantized}.

We review the main ideas of the multivariate setting. Considering
a function $f:[0,1]^{d}\ra\R$, we now place the vector variable $\bx\in[0,1]^{d}$
in bijection with sequences of the form $\bsig_{1},\bsig_{2},\ldots$,
where $\bsig_{k}=(\sigma_{k1},\ldots\sigma_{kd})\in\{0,1\}^{d}$ for
each $k$, using the identification 
\begin{equation}
\bx=\sum_{k=1}^{\infty}2^{-k}\bsig_{k}=\left(\begin{array}{c}
0.\sigma_{11}\sigma_{21}\sigma_{31}\cdots\\
0.\sigma_{12}\sigma_{22}\sigma_{32}\cdots\\
\vdots\\
0.\sigma_{1d}\sigma_{2d}\sigma_{3d}\cdots
\end{array}\right),\label{eq:ident-1}
\end{equation}
 where the entries in the expression at right indicate binary decimal
expansions for each of the components of $\bx$.

We choose a depth $K$ at which to truncate the decimal expansion,
so the identification 
\[
\bx\leftrightarrow(\bsig_{1},\ldots,\bsig_{K})
\]
 is a bijection between the dyadic grid $\mathcal{D}_{K}^{d}=\left(2^{-K}\,\mathbb{Z}^{d}\right)\cap[0,1)^{d}$
and the set $\{0,1\}^{d}\times\cdots\times\{0,1\}^{d}$, where the
direct product includes $K$ factors. Based on this identification,
we can in turn identify functions $f:\mathcal{D}_{K}^{d}\ra\R$ with
tensors $T\in\R^{2^{d}}\times\cdots\times\R^{2^{d}}$ via 
\[
f(\bx)=T(\bsig_{1},\ldots,\bsig_{d}).
\]

Such a tensor could be compressed as a tensor train where the external
dimension of each core is $2^{d}$. However, it is typical not to
take this course but instead to further split each $\bsig_{k}=(\sigma_{k1},\ldots\sigma_{kd})$
into its component parts to allow for the possibility of additional
compression. 

Accordingly we can $f$ as a tensor via \textbf{\emph{either}} 
\begin{equation}
f(\bx)=T(\underbrace{\sigma_{11},\sigma_{12},\ldots,\sigma_{1d}}_{\text{depth 1}},\,\underbrace{\sigma_{21},\sigma_{22},\ldots,\sigma_{2d}}_{\text{depth 2}},\,\ldots,\underbrace{\sigma_{K1},\sigma_{K2},\ldots,\sigma_{Kd}}_{\text{depth K}})\label{eq:interleaving}
\end{equation}
 \textbf{\emph{or}} 
\begin{equation}
f(\bx)=T(\underbrace{\sigma_{11},\sigma_{21},\ldots,\sigma_{K1}}_{\text{variable 1}},\,\underbrace{\sigma_{12},\sigma_{22},\ldots,\sigma_{K2}}_{\text{variable 2}},\,\ldots,\underbrace{\sigma_{1d},\sigma_{2d},\ldots,\sigma_{Kd}}_{\text{variable d}}).\label{eq:serial}
\end{equation}
 In the first representation (\ref{eq:interleaving}), the bits $\sigma_{ki}$
are organized first by depth index $k$ and then by variable index
$i$. In the second representation (\ref{eq:serial}), they are organized
first by variable and then by depth.

Typically (\ref{eq:interleaving}) is referred to as the \textbf{\emph{interleaved
ordering}}, while (\ref{eq:serial}) is referred to as the \textbf{\emph{serial
ordering}}. (However, we comment that the naming convention depends
on a matter of perspective.) In the following we explain how our construction
generalizes to both of these orderings.

\subsection{Interleaved ordering \label{sec:interleaved}}

For simplicity of presentation, we will explicitly consider the case
where $\bx=(x,y,z)\in\R^{3}$, i.e., $d=3$.

Here the left core $A_{\mathrm{L}}^{x}\in\R^{2\times1\times(N+1)^{3}}$
is defined by 
\begin{equation}
\left[A_{\mathrm{L}}(\sigma)\right]^{1,\beta_{x}\beta_{y}\beta_{z}}=f\left(\frac{\sigma+c^{\beta_{x}}}{2},c^{\beta_{y}},c^{\beta_{z}}\right),\quad\beta_{x},\beta_{y},\beta_{z}\in[N+1],\ \sigma\in\{0,1\}.\label{eq:ALxyz}
\end{equation}
 Then with the univariate interpolating tensor core $A\in\R^{2\times(N+1)\times(N+1)}$
defined as in (\ref{eq:Aint}), we define interpolating tensor cores
$A^{x},A^{y},A^{z}\in\R^{2\times(N+1)^{3}\times(N+1)^{3}}$ for the
$x$, $y$, and $z$ dimensions as 
\begin{equation}
A^{x}(\sigma)=A(\sigma)\otimes I_{N+1}\otimes I_{N+1},\quad A^{y}(\sigma)=I_{N+1}\otimes A(\sigma)\otimes I_{N+1},\quad A^{z}(\sigma)=I_{N+1}\otimes I_{N+1}\otimes A(\sigma).\label{eq:Axyz}
\end{equation}
 Observe that appending the tensor core $A^{x}$ from the right, for
example, performs Chebyshev interpolation in the $x$ dimension.

Finally the tensor train is capped off with cores $A_{\mathrm{R}}^{x}\in\R^{2\times(N+1)^{3}\times(N+1)^{2}}$,
$A_{\mathrm{R}}^{y}\in\R^{2\times(N+1)^{2}\times(N+1)}$, $A_{\mathrm{R}}^{z}\in\R^{2\times(N+1)\times1}$
given by 
\begin{equation}
A_{\mathrm{R}}^{x}(\sigma)=A_{\mathrm{R}}(\sigma)\otimes I_{N+1}\otimes I_{N+1},\quad A_{\mathrm{R}}^{y}(\sigma)=A_{\mathrm{R}}(\sigma)\otimes I_{N+1},\quad A_{\mathrm{R}}^{z}(\sigma)=A_{\mathrm{R}},\label{eq:ARxyz}
\end{equation}
 where $A_{\mathrm{R}}$ is defined as in (\ref{eq:AR}). The overall
construction is illustrated in Figure \ref{fig:interleaving}.

\begin{figure}
\centering{}\includegraphics[scale=0.6]{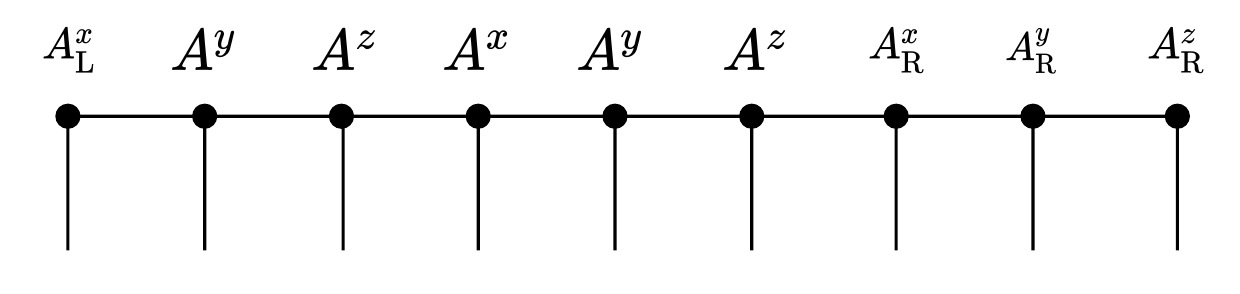}\caption{Interpolative multivariate QTT construction in the interleaved ordering.}
\label{fig:interleaving}
\end{figure}

It is straightforward to extend the rank-revealing construction of
Section \ref{sec:rankrevealing} to this setting by successively performing
SVDs in the same fashion. Note that due to the tensor product structure,
matrix-vector multiplication by the matrices $A^{x}(\sigma),A^{y}(\sigma),A^{z}(\sigma)$
can be achieved in $O(N^{4})$ time. In the general $d$-dimensional
setting, the scaling of these matvecs is $O(N^{d+1})$. Therefore
if the revealed rank is $r$, then the cost of this construction is
only $O(N^{d+1}r^{2})$.

It also automatic to extend the sparse interpolative construction
of Section \ref{sec:sparse} to this setting by simply replacing the
dense tensor core $A(\sigma)$ with its sparse counterpart. Then we
can obtain $O(N^{d}r^{2})$ complexity in the rank-revealing construction.

It is also possible to extend the construction of Section \ref{sec:sharp}
for resolving sharp features, by choosing a collection of nested dyadic
rectangles.

\subsection{Serial ordering \label{sec:serial}}

Again for concreteness we will explicitly consider the case where
$\bx=(x,y,z)\in\R^{3}$, i.e., $d=3$.

For QTT construction in serial ordering, we will make use of the tensor
cores $A_{\mathrm{L}}$, $A^{x}$, $A_{\mathrm{R}}^{x}$, $A_{\mathrm{R}}^{y}$,
and $A_{\mathrm{R}}^{z}$, defined before in equations (\ref{eq:ALxyz}),
(\ref{eq:Axyz}), and (\ref{eq:ARxyz}). However, the tensors $A^{y}$
and $A^{z}$ defined above are not directly suitable, so we also define
\[
\hat{A}^{y}(\sigma)=A(\sigma)\otimes I_{N+1},\quad\hat{A}^{z}(\sigma)=A(\sigma).
\]
 Here although simply $\hat{A}^{z}=A$, we opt for the suggestive
notation. The construction in serial ordering can be viewed as interpolating
in the $x$ variable to full depth, then capping off this interpolation
and interpolating in the $y$ variable to full depth, etc. This procedure
is illustrated in Figure \ref{fig:serial}.
\begin{figure}
\centering{}\includegraphics[scale=0.6]{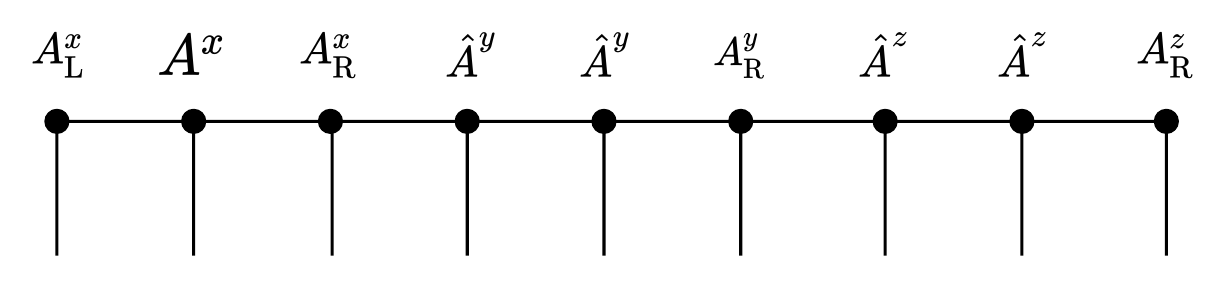}\caption{Interpolative multivariate QTT construction in the serial ordering.}
\label{fig:serial}
\end{figure}

The rank-revealing and sparse constructions of Sections \ref{sec:rankrevealing}
and \ref{sec:sparse} can be extended here to yield cost scalings
of $O(N^{d+1}r^{2})$ and $O(N^{d}r^{2})$, respectively.

However, we comment that it is not obvious how to extend the multiresolution
construction of Section \ref{sec:sharp} to this setting.

\section{Numerical experiments \label{sec:numerical}}

In this section we present several illustrative numerical experiments.

\subsection{Dense interpolation}

Consider the function $f:[0,1]\ra\R$ defined by 
\begin{equation}
f(x)=\sum_{j=1}^{J}\left[a_{j}\cos(2\pi jx)+b_{j}\sin(2\pi jx)\right],\label{eq:oscil}
\end{equation}
 where the $a_{j},b_{j}$ are independently distributed standard normal
random variables.

First we fix one typical instantiantiation of this function with $J=25$.
In Figure \ref{fig:oscillatory}, we present the accuracies of both
our basic interpolative construction and tensor cross interpolation
(implemented as $\texttt{amen\_cross}$ in the $\textsf{TT-Toolbox}$
package \cite{TTtoolbox}) against the number of function evaluations
required by each method. Note that the construction of Section \ref{sec:basicconstruction}
requires $2N+1$ function evaluations. The figure demonstrates the
significant advantage of the first approach.

\begin{figure}
\centering{}\includegraphics[scale=0.4]{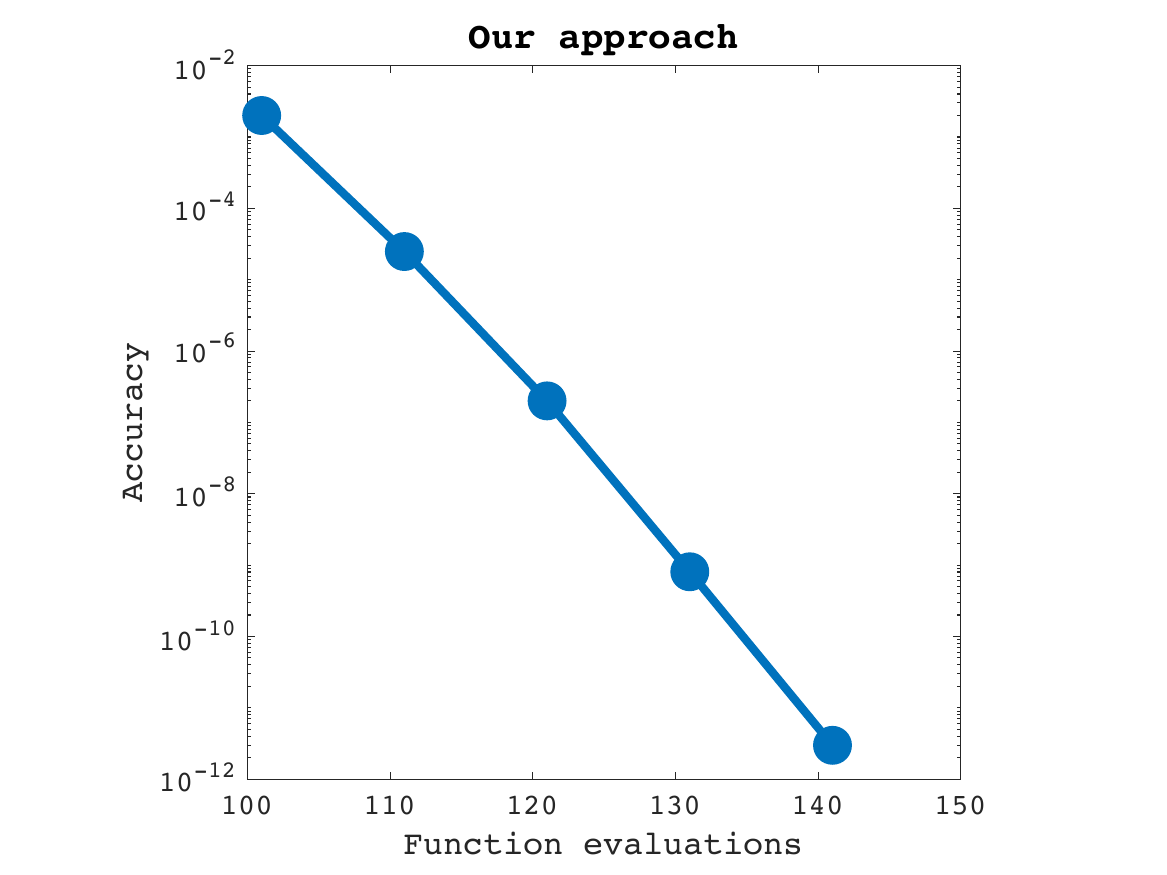}\includegraphics[scale=0.4]{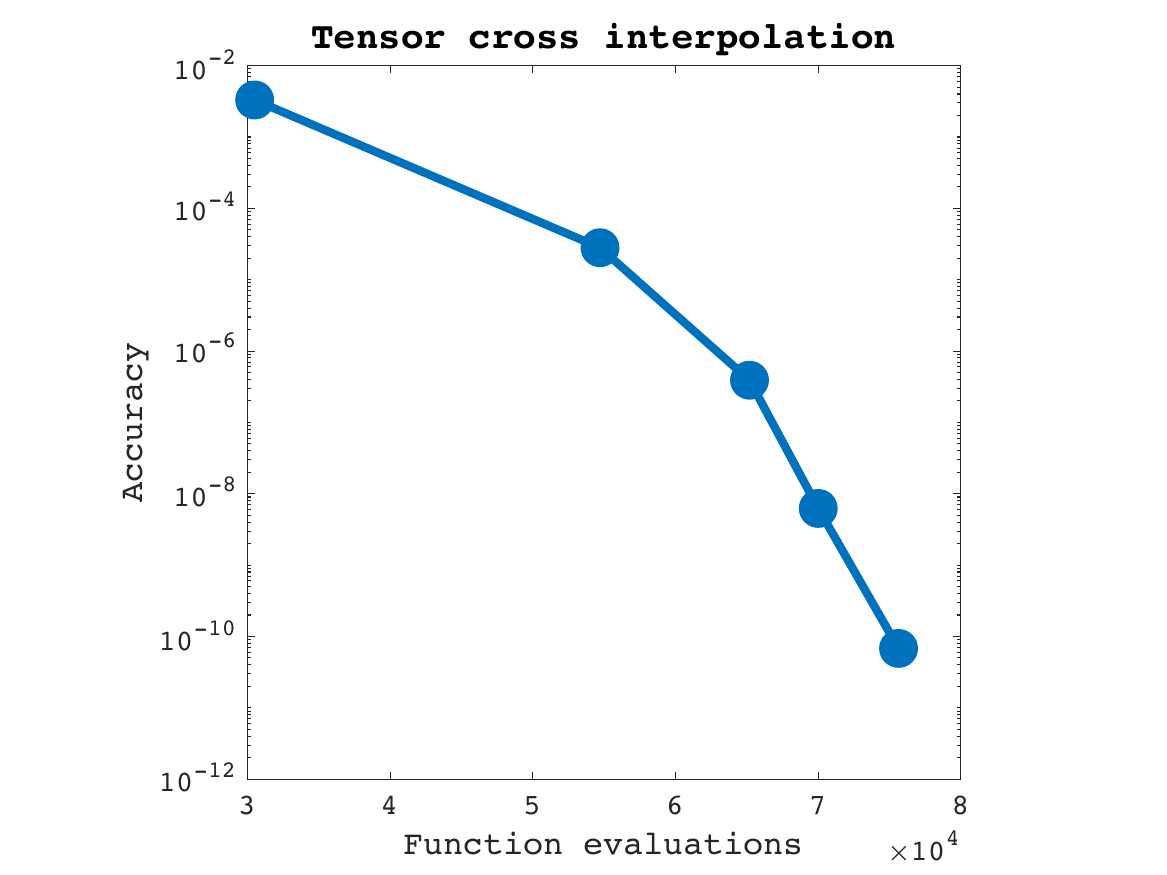}\caption{Error of our basic interpolative construction (left) and tensor cross
interpolation (right) for (\ref{eq:oscil}) with $J=25$, plotted
against the number of function evaluations. Error is measured by the
infinity norm on the dyadic grid $\mathcal{D}_{20}$. Note the different
scales of the horizontal axes.}
\label{fig:oscillatory}
\end{figure}

When $J$ becomes larger than about $100$, as the function becomes
highly oscillatory, tensor cross interpolation (TCI) fails to converge,
while our interpolative construction remains stable. We show in Figure
\ref{fig:oscillatory-1} that our accuracy remains roughly fixed if
we scale $N=\Omega(J)$.

\begin{figure}
\vspace{5mm}
\centering{}%
\begin{tabular}{|c|c|c|c|c|c|c|c|}
\hline 
$J$ & 200 & 300 & 400 & 500 & 600 & 1000 & 2000\tabularnewline
\hline 
\hline 
Error & {\scriptsize{}$9.8\times10^{-11}$} & {\scriptsize{}$1.1\times10^{-10}$} & {\scriptsize{}$8.4\times10^{-11}$} & {\scriptsize{}$1.3\times10^{-10}$} & {\scriptsize{}$1.8\times10^{-10}$} & {\scriptsize{}$2.2\times10^{-10}$} & {\scriptsize{}$3.5\times10^{-10}$}\tabularnewline
\hline 
\end{tabular}\caption{Error of our approach for (\ref{eq:oscil}) as a function of $J$,
where we take $N=2J$. Error is measured by the infinity norm on the
dyadic grid $\mathcal{D}_{20}$. Note that tensor cross interpolation
fails to converge for these examples.}
\label{fig:oscillatory-1}
\end{figure}

\subsection{Sparse interpolation}

For $\alpha>0$, consider the function 
\begin{equation}
f(x)=\frac{\alpha}{\sqrt{\alpha^{2}+(x-1/2)^{2}}},\label{eq:alpha}
\end{equation}
 which is peaked at $x=1/2$. The peak becomes increasingly sharp
as $\alpha\ra0$. However (cf. Section \ref{sec:sharp}), the TT ranks
of $f$ remain roughly constant in this limit.

We will demonstrate that the sparse and rank-revealing interpolation
scheme of Section \ref{sec:sparse} can be applied effectively to
this function in the large $N$ limit. (Of course, the approach of
Section \ref{sec:sharp} could well be applied here given \emph{a
priori} knowledge about the peak location.)

In order to maintain fixed accuracy, we must scale $N$ with $\alpha$
as $N\sim C/\sqrt{\alpha}$. Alternatively, we can consider $\alpha$
as a function of $N$, for several different fixed values of $C$.
In Figure \ref{fig:alpha} we plot the error of our sparse interpolation
scheme. Given the necessary function evaluations, each construction
in the figure was completed in less than 0.5 seconds on a 2021 M1
MacBook Pro, which would not be possible of $\Omega(N^{3})$ or even
$\Omega(N^{2})$ operations were required.

\begin{figure}
\centering{}$\quad\quad\quad\quad$\includegraphics[scale=0.5]{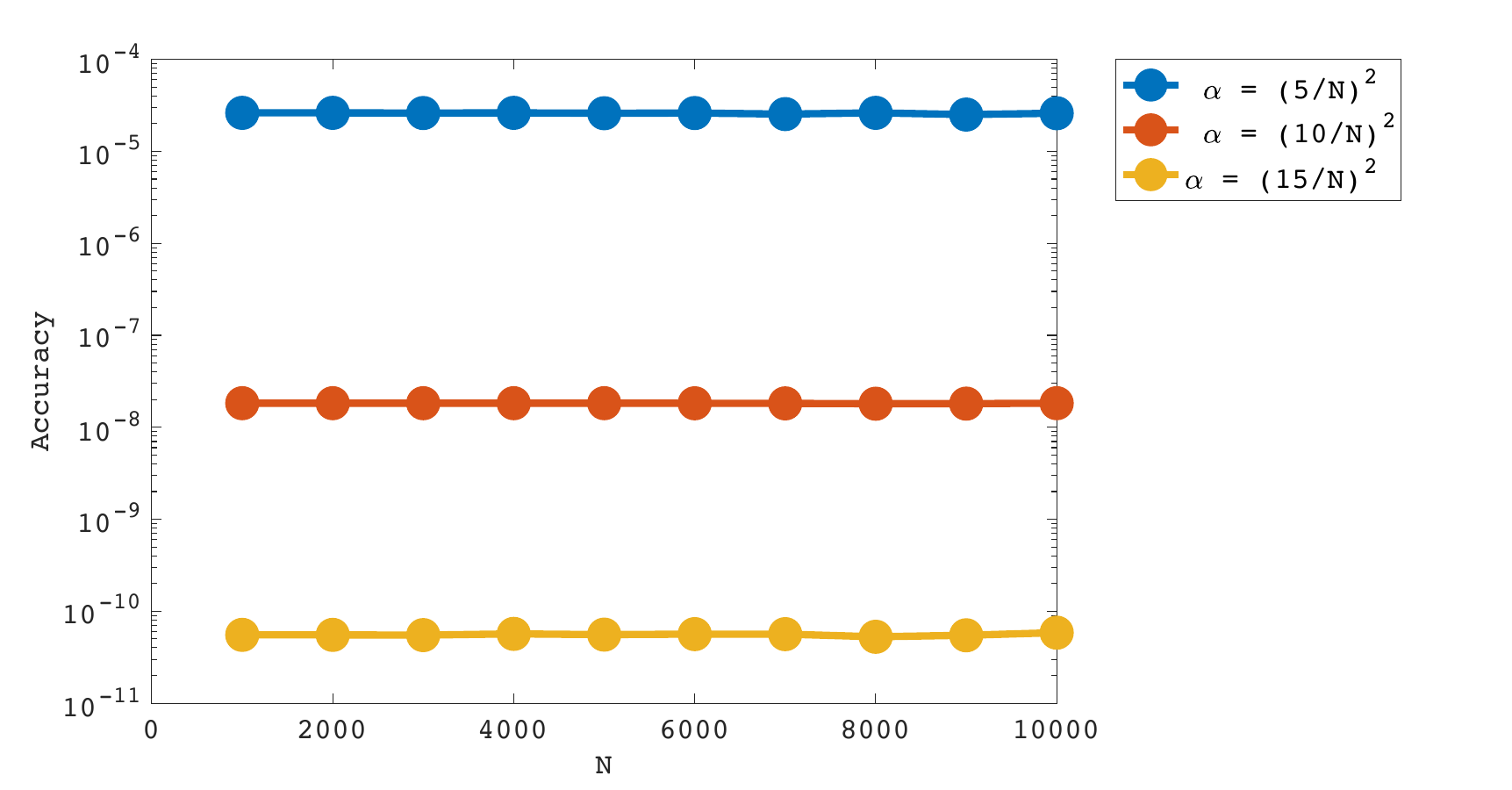}\caption{Error of our approach for (\ref{eq:alpha}) where we scale $\alpha$
with $N$ according to the legend. Error is measured by the infinity
norm on the dyadic grid $\mathcal{D}_{25}$.}
\label{fig:alpha}
\end{figure}

\subsection{Inverting the construction}

Consider (\ref{eq:alpha}) once again where we now fix the value $\alpha=0.1$.
We will use this example to validate the `inversion' algorithm of
Section \ref{sec:invert}. We choose $N=300$, more than large enough
to ensure machine precision of our QTT.

For the inverse construction, we simply take $q=1$, so the local
Lagrange interpolation is linear, which is accurate up to error $O(h^{2})$
on an interval of size $O(h)$. Therefore we expect that as the depth
$K$ is increased, the accuracy of our inversion algorithm should
scale as $(2^{-K})^{2}=2^{-2K}$. To measure this accuracy, we use
the algorithm of Section \ref{sec:invert} to recover the evaluations
$f(c^{\alpha}/2)$, $\alpha=0,\ldots,N$, and we record the worst
case error over $\alpha$ of the recovery. The results are plotted
in Figure \ref{fig:zip}, and they validate our scaling prediction.

\begin{figure}
\centering{}$\quad\quad$$\quad\quad\quad\quad$\includegraphics[scale=0.5]{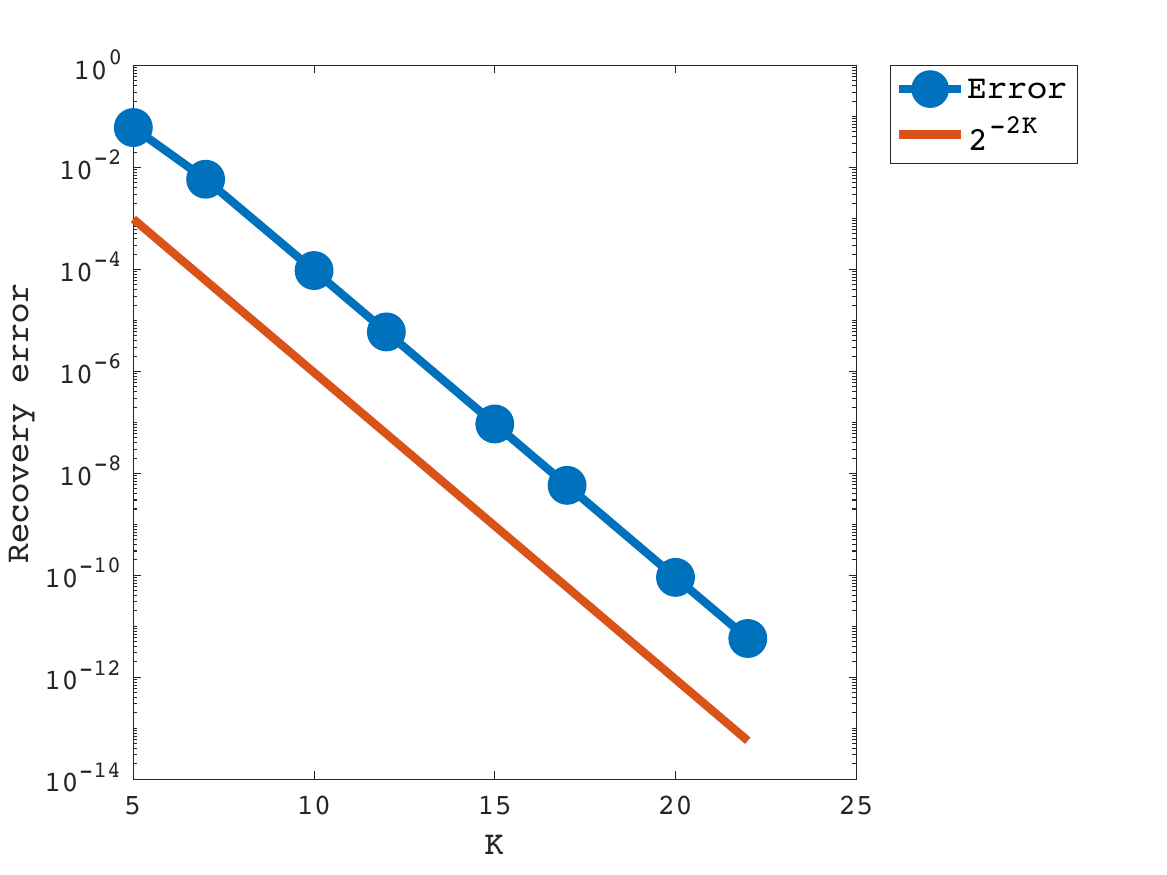}\caption{Worst-case error of the recovery of $f(c^{\alpha}/2)$ from a QTT
for $f$ (following the algorithm of Section \ref{sec:invert}), plotted
against the depth $K$.}
\label{fig:zip}
\end{figure}

\subsection{Multiresolution construction}

Next consider the Gaussian function 
\begin{equation}
f(x)=e^{-\frac{1}{2}(x/\alpha)^{2}}.\label{eq:gaussian}
\end{equation}
In \cite{doi:10.1137/120864210} a bound for the QTT ranks of $f$
was provided based on approximation with Fourier series, but this
rank bound is not uniform with respect to the scale $\alpha$ of the
Gaussian function. In particular, in the limit $\alpha\ra0$, the
quantitative smoothness of $f$ deteriorates. Nevertheless, it may
be observed empirically that the QTT ranks of $f$ are bounded independently
of $\alpha$. Our multiresolution construction of Section \ref{sec:sharp}
clarifies this phenomenon, and here we validate this perspective with
numerics.

Fix $K=25$, and adopt the subintervals considered above in our discussion
of the function $f(x)=\sqrt{x}$ (cf. Section \ref{sec:sharp}), i.e.,
take $q_{k}=1$ and $\sigma_{1:k}^{1}=(0,0,\ldots,0)$ for all $k=1,\ldots,K-1$.
In Figure \ref{fig:gaussian}, we plot the error of the multiresolution
construction of Section \ref{sec:sharp} as a function of $\alpha$
for several values of $N$. Observe that uniformly bounded error is
achieved in the limit $\alpha\ra0$, and in fact machine precision
across all $\alpha$ is already almost attained by $N=18$.

\begin{figure}
\centering{}\includegraphics[scale=0.5]{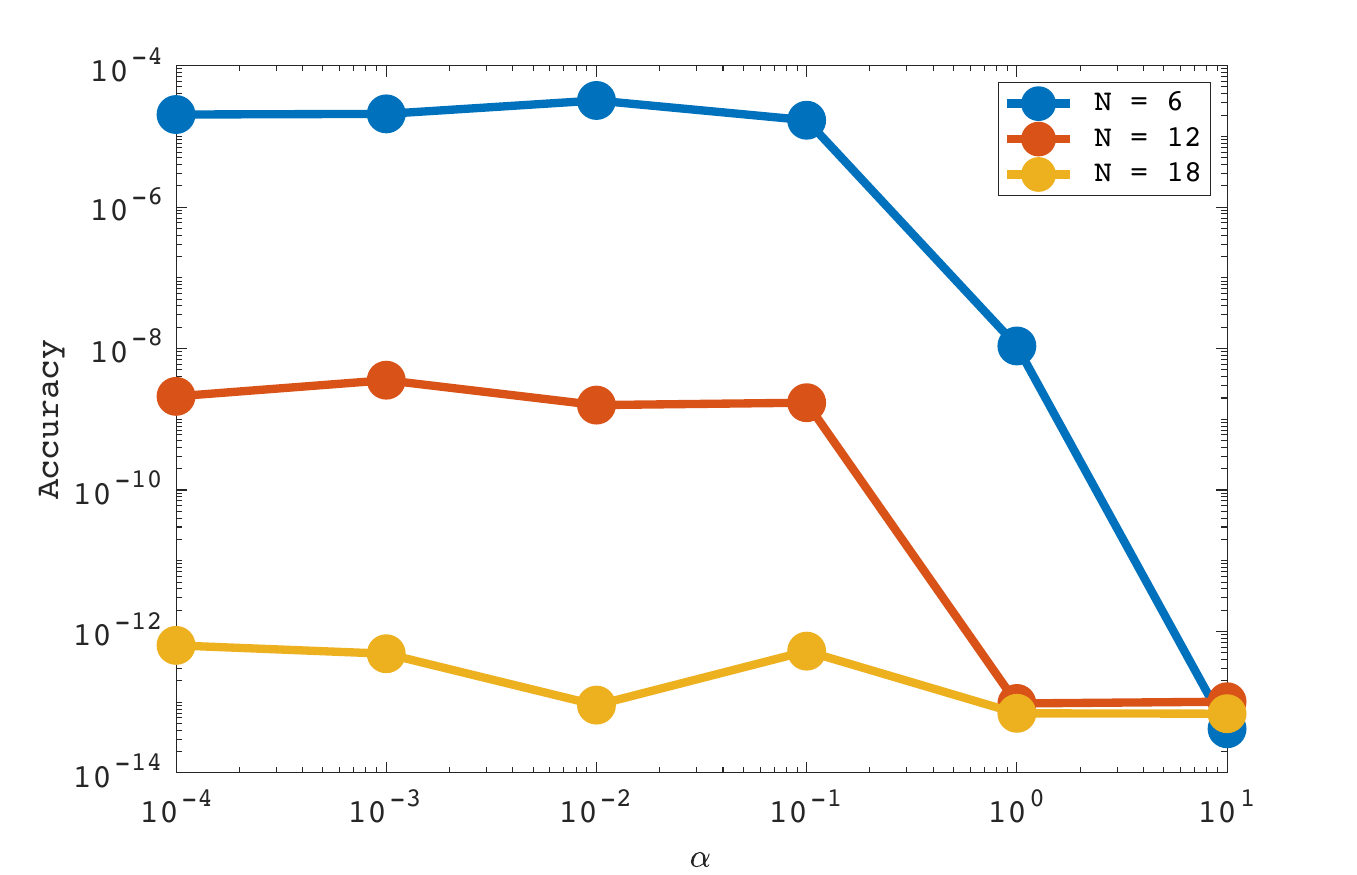}\caption{Error of the multiresolution construction for (\ref{eq:gaussian})
as a function of $\alpha$, for several values of $N$ indicated by
the legend. Error is measured by the infinity norm on the dyadic grid
$\mathcal{D}_{25}$.}
\label{fig:gaussian}
\end{figure}

\subsection{Multivariate construction}

Finally we simply demonstrate the application of the multivariate
construction in the serial ordering (cf. Section (\ref{sec:serial}))
to the bivariate function 
\begin{equation}
f(x,y)=\frac{1}{1+100\left[(x-1/2)^{2}+(y-1/2)^{2}\right]}.\label{eq:bivariate}
\end{equation}
 We fix $K=10$ and plot the error as a function of $N$ in Figure
\ref{fig:multi}.

\begin{figure}
\centering{}\includegraphics[scale=0.5]{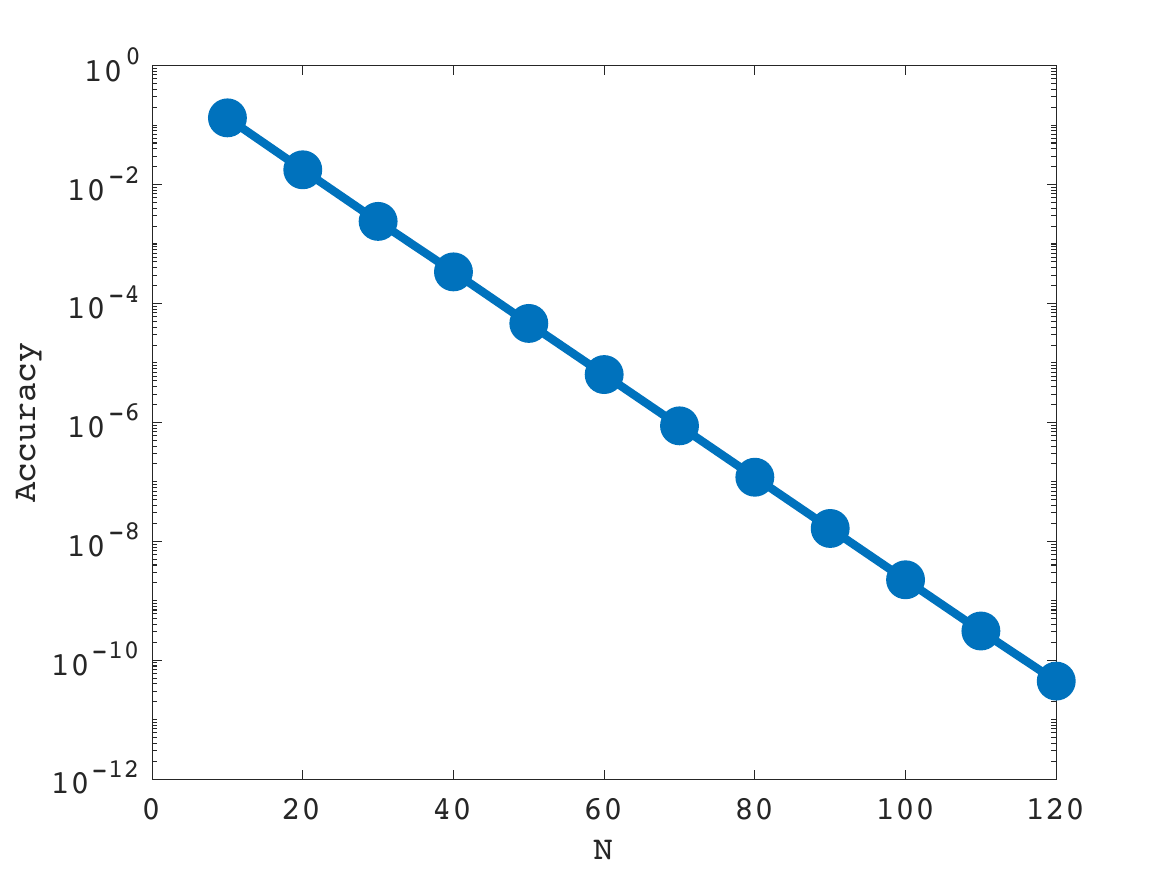}\caption{Error of the multivariate construction in the serial ordering (cf.
Section (\ref{sec:serial})) for the function (\ref{eq:bivariate})
as a function of $N$. Error is measured by the infinity norm on the
dyadic grid $\mathcal{D}_{10}^{2}$.}
\label{fig:multi}
\end{figure}

\bibliographystyle{plain}
\bibliography{bigbib}

\appendix

\part*{\protect\pagebreak}

\section*{Appendix A $\ \ $Proofs of interpolation bounds \label{app:interp}}

In this section we prove the interpolation bounds of Section \ref{sec:interp}.
\begin{proof}
[Proof of Proposition \ref{prop:diffinterp}.] The bound on $E_{m,N}[f]$
follows directly from the interpolation error bound of Theorem 7.2
of \cite{TrefethenBook2019}, which bounds the pointwise interpolation
error of a function $g:[-1,1]\ra\R$ (using a suitable Chebyshev-Lobatto
grid) by 
\[
\frac{4V}{\pi p(N-p)^{p}},
\]
 where $V$ is the total variation of $g^{(p)}$.

The only wrinkle is that we need to shift and scale the functions
that we want to interpolate to the reference interval $[-1,1]$. At
level $m$, for any fixed $u\in[0,1-2^{-m}]$, we want to interpolate
the function $g_{m,u}(x):=f\left(u+2^{-m}\left[\frac{x+1}{2}\right]\right)$
defined on the domain $x\in[-1,1]$. Since $g_{m,u}^{(p)}$ is differentiable
by assumption, the total variation of $g_{m,u}^{(p)}$ is equal to
\[
\Vert g_{m,u}^{(p+1)}\Vert_{L^{1}([-1,1])}\leq2\Vert g_{m,u}^{(p+1)}\Vert_{L^{\infty}([-1,1])}.
\]
 But $g_{m,u}^{(p+1)}(x)=2^{-(m+1)}f^{(p+1)}\left(u+2^{-m}\left[\frac{x+1}{2}\right]\right)$,
so 
\[
\Vert g_{m,u}^{(p+1)}\Vert_{L^{\infty}([-1,1])}\leq2^{-(m+1)}\Vert f^{(p+1)}\Vert_{L^{\infty}([0,1])}\leq2^{-(m+1)}C,
\]
 and it follows that the total variation of $g_{m,u}^{(p)}$ is bounded
by $2^{-m}C$. This concludes the proof.
\end{proof}
$ $
\begin{proof}
[Proof of Proposition \ref{prop:analyticinterp}.] The bound on $E_{m,N}[f]$
follows directly from the interpolation error bound of Theorem 8.2
of \cite{TrefethenBook2019}, which states that if $g:[-1,1]\ra\R$
extends analytically to $\mathcal{E}_{\rho}$ and $\vert g\vert\leq M$
on $\mathcal{E}_{\rho}$, then the pointwise interpolation error for
$g$ using Chebyshev-Lobatto grid of size $N$ is bounded by 
\[
\frac{4M\rho^{-N}}{\rho-1}.
\]

At level $m$, for any fixed $u\in[0,1-2^{-m}]$, we want to interpolate
the function $g_{m,u}(x):=f\left(u+2^{-m}\left[\frac{x+1}{2}\right]\right)$
defined on the domain $x\in[-1,1]$. Hence we are motivated to consider
the question: for which values of $\rho'>1$ does the containment
\[
u+2^{-m}\left[\frac{\mathcal{E}_{\rho'}+1}{2}\right]\subset\frac{\mathcal{E}_{\rho}+1}{2}
\]
 hold? Equivalently, we ask for the containment
\[
\left(2u-1+2^{-m}\right)+2^{-m}\mathcal{E}_{\rho'}\subset\mathcal{E}_{\rho}.
\]
 Note that the extreme values for $2u-1+2^{-m}$ (over all allowed
values $u$) are $\pm(1-2^{-m})$. Therefore it will suffice to ask
for 
\[
(1-2^{-m})+2^{-m}\mathcal{E}_{\rho'}\subset\mathcal{E}_{\rho}.
\]
 In turn, for this containment of ellipses to hold it suffices that
the following inequalities hold for their semi-major axes: 
\[
(1-2^{-m})+2^{-m}a_{\rho'}\leq a_{\rho},\quad2^{-m}b_{\rho'}\leq b_{\rho},
\]
 which finally lead to the desiderata 
\[
a_{\rho'}\leq2^{m}\left(a_{\rho}-1\right)+1,\quad b_{\rho'}\leq2^{m}b_{\rho}.
\]

First we observe that these inequalities always hold trivially if
we take $\rho'=\rho$. However, asymptotically as $m\ra\infty$, we
will see that we can achieve far larger $\rho'=\Omega(2^{m})$.

Indeed, observe that for all $r\geq1$, $a_{r}=\frac{r+r^{-1}}{2}$
and $b_{r}=\frac{r-r^{-1}}{2}$ satisfy the inequalities 
\[
\frac{r}{2}\leq a_{r}\leq\frac{r+1}{2},\quad\frac{r-1}{2}\leq b_{r}\leq\frac{r}{2}.
\]
 Therefore it suffices that 
\[
\rho'\leq2^{m+1}\left(a_{\rho}-1\right)+1,\quad\rho'\leq2^{m}(\rho-1).
\]
 In fact, $a_{\rho}-1\leq\frac{\rho-1}{2}$ for $\rho\geq1$, so both
conditions are implied by the stronger inequality 
\[
\rho'\leq2^{m+1}(a_{\rho}-1)=2^{m}\frac{(\rho-1)^{2}}{\rho}.
\]

Therefore if we take 
\[
\rho_{m}:=\max\left[\rho,\,2^{m}\frac{(\rho-1)^{2}}{\rho}\right],
\]
 then $g_{m,u}$ is analytic on $\mathcal{E}_{\rho_{m}}$, and moreover
$\vert g_{m,u}\vert$ is bounded by $B$ on $\mathcal{E}_{\rho_{m}}$.
The result then follows from Theorem 8.2 of \cite{TrefethenBook2019}.
\end{proof}
$ $
\begin{proof}
[Proof of Proposition \ref{prop:bandlimitedinterp}.] At level $m$,
for any fixed $u\in[0,1-2^{-m}]$, we want to interpolate the function
$g_{m,u}(x):=f\left(u+2^{-m}\left[\frac{x+1}{2}\right]\right)$ defined
on the domain $x\in[-1,1]$. As in the proof of Proposition \ref{prop:analyticinterp},
we will make use of Theorem 8.2 of \cite{TrefethenBook2019}, which
bounds the error of Chebyshev interpolation for analytic functions.

Now $f$ extends analytically to the entire complex plane via the
formula $f(z)=\frac{1}{2\pi}\int e^{i\omega x}\,d\mu(\omega)$ for
$z\in\mathbb{C}$. Thus $g_{m,u}$ extends as 
\[
g_{m,u}(z)=\frac{e^{i\omega\left(u+2^{-(m+1)}\right)}}{2\pi}\int e^{i2^{-(m+1)}\omega z}\,d\mu(\omega).
\]
 Since $\mu$ is supported on $[-\Omega,\Omega]$, we have that $\vert e^{i2^{-(m+1)}\omega z}\vert\leq e^{2^{-(m+1)}\Omega\,\Im(z)}$
on the support of $\mu$, and therefore
\[
\vert g_{m,u}(z)\vert\leq\frac{\vert\mu\vert}{2\pi}\,e^{2^{-(m+1)}\Omega\,\Im(z)}.
\]
 Now for any $z$ in the Bernstein ellipse $\mathcal{E}_{\rho}$,
we have $\mathrm{Im}(z)\leq\frac{\rho-\rho^{-1}}{2}\leq\frac{\rho}{2}$.
Therefore 
\[
\vert g_{m,u}(z)\vert\leq M:=\frac{\vert\mu\vert}{2\pi}\,e^{2^{-(m+2)}\Omega\rho}
\]
 for all $z\in\mathcal{E}_{\rho}$.

Then Theorem 8.2 of \cite{TrefethenBook2019} bounds the interpolation
error for $g_{m,u}$ by 
\[
\frac{4M\rho^{-N}}{\rho-1}\leq\frac{2\vert\mu\vert}{\pi}\,\frac{\exp\left[2^{-(m+2)}\Omega\rho-N\log\rho\right]}{\rho-1},
\]
 where $\rho>1$ is arbitrary. Simply take $\rho=2$, yielding 
\[
E_{m,N}[f]\leq\frac{2\vert\mu\vert}{\pi}\exp\left[2^{-(m+1)}\Omega-N\log2\right].
\]
 Since $\log2\geq\frac{1}{2}$, we have 
\[
E_{m,N}[f]\leq\frac{2\vert\mu\vert}{\pi}\exp\left[\frac{1}{2}\left(2^{-m}\Omega-N\right)\right],
\]
 as was to be shown.
\end{proof}

\end{document}